\documentclass[11pt]{article}

\usepackage{amssymb, amsmath,amsmath,latexsym,amssymb,amsfonts,amsbsy, amsthm,mathtools,graphicx,color,cases,bm}
\usepackage[english]{babel}
\usepackage{indentfirst}
\usepackage{graphicx,float}
\usepackage{xcolor,txfonts}
\numberwithin{equation}{section}

\newtheorem{thm}{Theorem}[section]
\newtheorem{lem}{Lemma}[section]

\newtheorem{rem}{Remark}[section]

\newtheorem{defn}{Definition}[section]
\theoremstyle{definition}
\theoremstyle{remark}

\def\dd{{\rm d}}

\begin{document}
\title{\bf Blow-Up Theory and Liouville-Type Theorem for Solutions of a Class of Generalized Camassa-Holm-Kadomtsev-Petviashvili Equations}
\author{Xueli Ke$^a$\thanks{E-mail: kexueli@hpu.edu.cn},
Jiamin Wang$^b$ \thanks{E-mail: wjm1061088@126.com},
Aibin Zang$^c$ \thanks{E-mail: zangab05@126.com}\thanks{Corresponding author} \\
\textit{\small  a. School of Mathematics and Information Science, Henan Polytechnic University, Jiaozuo, Henan 454003, P. R. China} \\
\textit{\small  b.School of Primary Education, Yichun Early Childhood Teachers' College, Yichun, Jiangxi 336000, P. R. China}\\
\textit{\small  c. School of Artificial Intelligence and Information Engineering, Yichun University, Yichun, Jiangxi 336000, P. R. China}
}

\date{}
\maketitle

\begin{abstract}
	We investigate the blow-up behavior and Liouville-type theorems of solutions to a class of generalized Camassa-Holm-Kadomtsev-Petviashvili (CH-KP) equations with a generally smooth nonlinear term $g(u)$. First, using the continuation method, we  establish a blow-up criterion that is independent of the regularity index of initial data. Under the assumption that  \( g'(u) \) is uniformly bounded, we prove the blow-up theorem and a weighted blow-up result by means of characteristic lines, a priori estimates and the Riccati inequality. Moreover, we extend these blow-up results to the setting where  \( g'(u) \) is polynomially controlled, which includes typical nonlinearities such as \( g(u)=\kappa u+3u^2 \) for the classical CH-KP equations.  Furthermore, a Liouville-type uniqueness theorem is established under the condition $g(u) \geq \gamma u^2$  with $u \neq 0$,  $g(u)>\gamma u^2$.

{\textbf{Keywords:} Generalized Camassa-Holm-Kadomtsev-Petviashvili equation; Blow-up criterion; Liouville-type theorem}

{\textbf{AMS Subject Classification (2020):} 35B35, 35B40, 35Q35, 76D03}
\end{abstract}

\section{Introduction}\label{S0}
The study of water waves has been a popular and fascinating subject in recent years. One of the most celebrated models is the Korteweg-de Vries (KdV) equation\cite{KD1895}, which describes unidirectional, shallow-water waves and is given by
\begin{align}\label{eq:01}
u_t+u_x+\frac{3}{2}uu_x+\frac{1}{6}u_{xxx}=0.
\end{align}
In two dimensions, the KdV equation naturally extends to the Kadomtsev-Petviashvili (KP) equation \cite{KP1970}, which incorporates transverse effects:
\begin{align}\label{eq:02}
(u_t+u_{xxx}+u_x+uu_x)_{x}+u_{yy}=0, t>0, (x,y)\in\mathbb{R}^2,
\end{align}
with initial condition $u|_{t=0}=u_0(x,y),\:(x,y)\in\mathbb{R}^2$. The well-posedness of equation \eqref{eq:02}, which generalizes the KdV equation by considering weak diffraction, has been extensively studied over the past three decades. While the KP equation captures the propagation of weakly dispersive waves in two dimensions \cite{HHK2009,MST2011}, other models are required to describe wave breaking. A prominent example is the Camassa-Holm (CH) equation \cite{CH1993,CHK2005, CL2009,JOHNSON2002}, that is
\begin{align}\label{eq:03}
u_t-u_{xxt}+2\kappa u_x+3uu_x=2u_xu_{xx}+uu_{xxx},
\end{align}
where $\kappa$ is a parameter. The CH equation has garnered significant attention due to its integrable structure and its admission of peakon solutions \cite{CL2009,JOHNSON2002}.

A natural two-dimensional counterpart of the CH equation, analogous to the relationship between the KdV and KP equations, is the Camassa-Holm-Kadomtsev-Petviashvili (CH-KP) equation. As pointed in \cite{GLL2021}, it takes the form
\begin{align}\label{eq:1}
{{\left( {{u}_{t}}-\frac{5}{12}\mu{{u}_{xxt}}+u_x+\frac{3}{2}\varepsilon u{{u}_{x}}-\frac{1}{4}\mu u_{xxx}-
\frac{5}{24}\mu\varepsilon\left( 2{{u}_{x}}{{u}_{xx}}+u{{u}_{xxx}} \right) \right)}_{x}}+\frac{1}{2}\varepsilon^3{{u}_{yy}}=0.
\end{align}
where $u$ represents the horizontal velocity field at height $h_0=\frac{1}{\sqrt{2}}$ in flow domain, $\varepsilon=a/h_0\ll 1$ and $\mu=h_0^2/\lambda^2\ll 1$ are amplitude and shallowness parameters, respectively, where $a$ denotes the small amplitude and $\lambda$ the long wavelength.

Through suitable scaling and Galilean transformations with a parameter $\kappa\in\mathbb{R}$, eqation\eqref{eq:1} can be simplified to the standard form \cite{GLL2021}
\begin{align}\label{CHKP}
{{\left( {{u}_{t}}-u_{xxt}+\kappa u_x+3 u{{u}_{x}}- \left( 2{{u}_{x}}{{u}_{xx}}+u{{u}_{xxx}} \right) \right)}_{x}}+{{u}_{yy}}=0.
\end{align}
This equation was formally derived by Gui \textit{et al.}\cite{GLL2021} from the governing equations for three-dimensional incompressible and irrotational water waves in the long-wave limit.

While significant progress has been made on the local well-posedness, blow-up phenomena, and solitary wave solutions of equation \eqref{CHKP} \cite{CJ2021,GLL2021,MOON2022,NSW2024,PXZ2024}, the question of global existence for arbitrary initial data remains open due to its nonlocal structure and weak linear dispersion.

In this paper, we consider a further generalization of \eqref{CHKP} with a general smooth nonlinear term. By replacing the quadratic nonlinearity $3uu_x$ with a general function $g(u)$, we study the following initial value problem:
\begin{align}\label{GCHKP}
\left( u_t - u_{xxt} + \left( \frac{g(u)}{2} \right)_x - \gamma \left( 2u_x u_{xx} + u u_{xxx} \right) \right)_x + u_{yy} = 0.
\end{align}

Here, $u(t,x,y)$ represents the fluid velocity, $g\in C^{\infty}(\mathbb{R})$, with $g(0)=0$, and $\gamma>0$ is a constant. This formulation allows us to explore the impact of a broader class of nonlinearities on the dynamic behavior of the solution. Note that \eqref{CHKP} is recovered by setting
$g(u)=2\kappa u+3u^2$ and $\gamma=1$. Preliminary investigations into various forms of the CH-KP and generalized CH-KP equations have been reported in \cite{BISWAS2009,EFT2012,KB2011,LDJ2019,LZG2025,Moon2025,WKZ2022} and references therein.

Defining $G(x)=\frac12e^{-|x|},x\in\mathbb{R}$, we have the relation $\left(1-\partial_x^2\right)^{-1}f=G*f$ for all  $f\in L^2(\mathbb{R})$ (see  \cite{GLL2021}). The Cauchy problem for equation \eqref{GCHKP} can then be reformulated as the following system
\begin{align}\label{eq:3}
\begin{cases}u_t+\gamma uu_x+G*\partial_x\left(\frac12\:g(u)+\frac\gamma2\:u_x^2-\frac\gamma2u^2\right)+G*v_y=0,\:(x,y)\in\mathbb{R}^2,\\
u_y=v_x,\\
u|_{t=0}=u_0(x,y),\:(x,y)\in\mathbb{R}^2.
\end{cases}
\end{align}

This paper focuses on the well-posedness of system \eqref{eq:3} within the following Hilbert space.
\begin{defn}
For $s>0$, the Hilbert space $X^{s}(\mathbb R^{2})$ is defined by
$$X^s=X^{s}(\mathbb R^{2}):=\left\{u\in H^{s}(\mathbb R^{2}) | \partial_{x}^{-1}u\in H^{s}(\mathbb R^{2}),\, \partial_{x}u\in H^{s}(\mathbb R^{2})\right\},$$
equipped with  the inner product
$$(u,v)_{X^{s}}:=(u,v)_{H^{s}}+(\partial_{x}^{-1}u,\partial_{x}^{-1}v)_{H^{s}}+(\partial_{x}u,\partial_{x}v)_{H^{s}},\ \forall\ u,\, v\ \in X^{s}(\mathbb R^{2}),$$
and the norm
$$\|u\|_{X^{s}(\mathbb R^{2})}:=(\|u\|_{H^{s}(\mathbb R^{2})}^{2}+\|\partial_{x}^{-1}u\|_{H^{s}(\mathbb R^{2})}^{2}+\|\partial_{x} u\|_{H^{s}(\mathbb R^{2})}^{2})^{\frac{1}{2}}~\mbox{ for}~ u\in X^{s}(\mathbb R^{2}),$$
where $\partial_{x}^{-1}u(x,y):=\mathcal{F}^{-1}((i\xi)^{-1}\mathcal{F}(u)(\xi,\eta))$, and $$H^s(\mathbb{R}^2)=\left\{u\in L^2(\mathbb{R}^2)\left|\int_{\mathbb{R}^2}(1+|\xi|^2+|\eta|^2)^s|\mathcal{F}(u)(\xi,\eta)|^2\dd\xi \dd\eta<+\infty\right.\right\}$$ with the norm $\|u\|^2_{H^s}=\int_{\mathbb{R}^2}(1+|\xi|^2+|\eta|^2)^s|\mathcal{F}(u)(\xi,\eta)|^2\dd\xi \dd\eta,$  $\mathcal{F}$ and $\mathcal{F}^{-1}$ denote the Fourier transform and inverse Fourier transform with respect to the variables $x,\ y$, respectively.
\end{defn}

In the previous work \cite{WKZ2022}, we established the blow-up theory for system \eqref{eq:3} in $X^{\mathrm{s}}$ for specific nonlinear terms such as $g(u)=u^3$ or $u^4$. That analysis,however, relied heavily on the specific structure of these nonlinearities. It is a subtle challenge for a direct generalization to arbitrary smooth functions $g(u)$. The primary goal of this paper is to extend these results to a much wider class of nonlinearities.   We begin to establish a general blow-up criterion under suitable initial conditions. Subsequently, using energy methods, we obtain the corresponding blow-up theorems for the nonlinearities satisfying a polynomial growth condition
\begin{align}\label{cofg}
|g'(\xi)|\le C_1|\xi|^\alpha+C_2,\ \xi\in\mathbb{R},\ \alpha\ge 0.
\end{align}
This generalization is significant as it encompasses many physically relevant models. In particular, it includes the CH-KP equation \eqref{CHKP} (with $g(u)=2\kappa u+3u^2$ and $\gamma=1$) as well as the power-law cases $g(u)=u^3 \mbox{or}~u^4$ studied in \cite{WKZ2022}.

In whole paper, we assume $ g \in C^{\infty}(\mathbb{R})\  \text{and} \ g(0) = 0 $ and establish the following blow-up results.


\begin{thm}\label{T1.1}
Let \(s \geq 2\) and \(u \in C([0,T_{u_0}^*);X^s(\mathbb{R}^2)) \cap C^1([0,T_{u_0}^*);X^{s-2}(\mathbb{R}^2))\) be the unique solution to the system of equations \eqref{eq:3} with initial condition $u_0 \in X^s(\mathbb{R}^2)$. If the maximal existence time $T_{u_0}^* < \infty$, then
\begin{align}\label{eq:4}
\int_0^{T_{u_0}^*}\left\|\nabla u(\tau)\right\|_{L^\infty(\mathbb{R}^2)}^2\mathrm{d}\tau=\infty.
\end{align}
\end{thm}

\begin{rem}
The blow-up criterion \eqref{eq:4} indicates that the blow-up time of the solution $T_{u_0}^*$ is independent of the regularity index $s$ of the initial value $u_0$ in the $X^s$ space.
\end{rem}

\begin{thm}\label{T1.2} Assume that the smooth function $g$ satisfies \eqref{cofg} and \(u_0 \in X^s(\mathbb{R}^2)\), where \(s\ge3\). Define $m_0 = \inf\limits_{(x,y)\in\mathbb{R}^2} \partial_x u_0(x,y)>-\infty.$ Suppose there exists a positive constant $K $ depending only on $\gamma$, $g$, and $\|u_0\|_{X^s}$ such that $m_0<-\sqrt{\frac{K}{\gamma}}$. Then there exists a time $T^*$ with
\begin{align*}
0<T^* \leq \frac{1}{2\sqrt{K\gamma}} \ln\left( \frac{\sqrt{\gamma}m_0 - \sqrt{K}}{\sqrt{\gamma}m_0 + \sqrt{K}} \right)<\infty,
\end{align*}
such that for any \(T < T^*\), the system \eqref{eq:3} has a unique solution \(u \in C([0,T];X^s(\mathbb{R}^2)) \cap C^1([0,T];X^{s-2}(\mathbb{R}^2))\), but $u$ satisfies
\begin{align}\label{eq:6}
\limsup_{t\to T^{*^{-}}} \left\|\partial_{x}u(t,\cdot,\cdot)\right\|_{L^{\infty}}=\infty.
\end{align}
\end{thm}

\begin{thm}\label{T1.3} Under assumption of $g$ in Theorem \ref{T1.2} and choose a function \(\varphi \geq 0\) and \(\varphi \in H^2(\mathbb{R})\) such that \(\int_{\mathbb{R}} \varphi \mathrm{d}y = 1\). Let the initial velocity  \(u_0 \in X^s(\mathbb{R}^2)\), \(s>2\), and there exists a point $x_0 \in \mathbb{R}$ satisfying $$\int_{\mathbb{R}} u_{0x}(x_0,y)\varphi(y)\mathrm{d}y < -\sqrt{\frac{2}{\gamma}}C_3,$$  where
\begin{align*}
C_3 = \left\{ C(g, \|u_0\|_{X^s})\|u_0\|_{X^s} + \frac{3}{2}\gamma E(u_0)\|\varphi\|_{L^{\infty}} + E^{\frac{1}{2}}(u_0)\|\varphi''\|_{L^2} \right\}^{1/2},
\end{align*}
here $C(g,\|u_0\|_{X^s})$ is a constant depending only on $g$ and $\|u_0\|_{X^s}$, and $E(u) = \frac{1}{2} \int_{\mathbb{R}} (u^2 + u_x^2) dy$ denotes the energy. Then there exists a time  $T_0$ with
\begin{align*}
0<T_0 \leq \frac{1}{\sqrt{2\gamma} C_3} \ln \frac{\sqrt{\gamma} \int_{\mathbb{R}} u_{0x}(x_0,y)\varphi(y)\mathrm{d}y - \sqrt{2}C_3}{\sqrt{\gamma} \int_{\mathbb{R}} u_{0x}(x_0,y)\varphi(y)\mathrm{d}y + \sqrt{2}C_3} < \infty,
\end{align*}
such that for any $T<T_0$, the  equations \eqref{eq:3} admits a unique solution $u \in C([0,T];X^s(\mathbb{R}^2)) \cap C^1([0,T];X^{s-2}(\mathbb{R}^2))$, but $u$ satisfies
\begin{align}\label{eq:9}
\lim_{t\to T^{-}_0} \int_{\mathbb{R}} \inf_{x\in\mathbb{R}} u_x(t,x,y)\varphi(y)\mathrm{d}y = -\infty.
\end{align}
\end{thm}

To achieve these results, we employ convolution estimates, energy estimates, and an analysis of Riccati-type dynamics along characteristics. We first overcome the technical challenges posed by the general nonlinearity to establish a  blow-up criterion (Theorem \ref{T1.1}).
Using the boundedness of $g'$ and the energy method under condition \eqref{cofg}, we then derive sufficient conditions for blow up of solutions(Theorems \ref{T1.2} and \ref{T1.3}). Finally, inspired by arguments for the CH equation \cite{LP2020} and CH-KP equations \cite{GLL2021}, we prove a Liouville-type property for  \eqref{GCHKP} under the assumption $g(u) \geq \gamma u^2$, and $g(u)>\gamma u^2$ for all $u \neq 0$. This theorem establishes that non-trivial solutions cannot vanish identically  on any non-empty open set, highlighting the uniqueness and non-locality of the solution's support.

The paper is organized as follows. Section \ref{S1} presents preliminary lemmas. Section \ref{S2} is dedicated to the proof of the blow-up criterion (Theorem \ref{T1.1}). In Section \ref{S3}, we first derive sufficient conditions for blow up of solutions under the assumption $g'\in L^\infty$ and subsequently extend these to the polynomial growth case, thereby proving Theorems \ref{T1.2} and \ref{T1.3}. Section \ref{S4} provides the proof of the Liouville-type theorem.



\section{Preliminary Lemmas}\label{S1}

\begin{lem}\label{L2.1}
\textbf{(a)} The convolution operator $G*: H^s(\mathbb{R}^2) \to H^{s}(\mathbb{R}^2)$ is bounded for any real number $s$. That is, for any $ f \in H^s(\mathbb{R}^2)$, we have
\begin{align}\label{eq:11}
\|G*f\|_{H^s(\mathbb{R}^2)}\leq \|f\|_{H^s(\mathbb{R}^2)}.
\end{align}

\textbf{(b)} The operator $(\partial_x G)*: L^{\infty}(\mathbb{R}^2) \to L^{\infty}(\mathbb{R}^2)$ is bounded and  and the identity $\|\partial_x(G*f)\|_{L^\infty(\mathbb{R}^2)} = \|(\partial_x G)*f\|_{L^\infty(\mathbb{R}^2)}$ holds. That is, for any $f \in L^\infty(\mathbb{R}^2)$, it follows
\begin{align}\label{eq:12}
\|\partial_x(G*f)\|_{L^\infty(\mathbb{R}^2)} = \|(\partial_x G)*f\|_{L^\infty(\mathbb{R}^2)} \leq  \|f\|_{L^\infty(\mathbb{R}^2)}.
\end{align}

 \textbf{(c)} The operator $\partial_x \circ (G*): H^s(\mathbb{R}^2) \to H^s(\mathbb{R}^2)$ is bounded for any real number $s$. Namely, for any $f \in H^s(\mathbb{R}^2)$,
\begin{align}\label{eq:13}
\|\partial_x(G*f)\|_{H^s(\mathbb{R}^2)} \leq  \|f\|_{H^s(\mathbb{R}^2)}.
\end{align}
\end{lem}

\begin{proof}
\textbf{(a)} For brevity, denote $\mathcal{F}(f)(\xi,\eta)$ by $\widehat{f}(\xi,\eta)$. Combining Fourier transform formula for the convolution $\widehat{G*f}(\xi,\eta)=\widehat{G}(\xi)\widehat{f}(\xi,\eta)$ and $\widehat{G}(\xi)=\frac{1}{1+\xi^2}$ yields
\begin{equation}\label{eq:14}
\begin{aligned}
\| G * f\|_{H^{s}(\mathbb{R}^2)}^2 &= \int_{\mathbb{R}^2} \left| \widehat{(G * f)}(\xi,\eta) \right|^2 (1+\xi^2+\eta^2)^{s} \dd\xi \dd\eta \\
&= \int_{\mathbb{R}^2} \left| \frac{1}{1+\xi^2} \hat{f}(\xi,\eta) \right|^2 (1+\xi^2+\eta^2)^{s} \dd\xi \dd\eta \\
&\leq \int_{\mathbb{R}^2} \left|  \hat{f}(\xi,\eta) \right|^2 (1+\xi^2+\eta^2)^{s} \dd\xi \dd\eta = \|f\|_{H^{s}(\mathbb{R}^2)}^2 .
\end{aligned}
\end{equation}
Therefore, it holds that
\begin{align}\label{eq:15}
\| G * f\|_{H^{s}} \leq  \| f\|_{H^s}.
\end{align}

\textbf{(b)} From $G(x)=\frac{1}{2}e^{-|x|}$, we can compute that $\partial_x G(x) = -\frac{1}{2}\operatorname{sgn}(x)e^{-|x|}$ for $x\in \mathbb{R}$. It then follows that $\|\partial_x G\|_{L^1(\mathbb{R})} = \int_{\mathbb{R}} \left| -\frac{1}{2}\operatorname{sgn}(x)e^{-|x|} \right| \dd x = \frac{1}{2}\int_{\mathbb{R}} e^{-|x|} \dd x = 1 $. The two-dimensional convolution is defined as
\begin{align}\label{eq:16}
(\partial_x G * f)(x,y) = \int_{\mathbb{R}} \partial_x G(x-\xi)f(\xi,y)\dd\xi.
\end{align}
According to the Young inequality for convolutions, we obtain, for any $(x,y)\in\mathbb{R}^2$,
\begin{equation}\label{eq:17}
\begin{aligned}
|(\partial_x G * f)(x,y)| &\leq \int_{\mathbb{R}} |\partial_x G(x-\xi)| \cdot |f(\xi,y)| \dd\xi\\
& \leq \|f(\cdot,y)\|_{L^\infty} \|\partial_x G\|_{L^1(\mathbb{R})} = \|f(\cdot,y)\|_{L^\infty}.
\end{aligned}
\end{equation}
One yields
\begin{align}\label{eq:18}
\sup_{x,y} |(\partial_x G * f)(x,y)| \leq \sup_y \|f(\cdot,y)\|_{L^\infty_x} \leq \sup_{x,y} |f(x,y)| = \|f\|_{L^\infty(\mathbb{R}^2)}.
\end{align}

Moreover, since
\begin{align}\label{eq:20}
(G * f)(x,y) = \int_{\mathbb{R}} G(x-\xi)f(\xi,y)\dd\xi,
\end{align}
it follows that
\begin{align}\label{eq:21}
\partial_x (G * f)(x,y) = \partial_x \left( \int_{\mathbb{R}} G(x-\xi)f(\xi,y)\dd\xi \right) = \int_{\mathbb{R}} \partial_x G(x-\xi)f(\xi,y)\dd\xi = \partial_x G * f.
\end{align}
Therefore, we have
\[
\|\partial_x (G * f)\|_{L^\infty(\mathbb{R}^2)} = \|\partial_x G * f\|_{L^\infty(\mathbb{R}^2)} \leq \|f\|_{L^\infty(\mathbb{R}^2)}.
\]

\textbf{(c)} From the  properties of the Fourier transform \cite{EVANS2010}, it follows that

\begin{equation}\label{eq:22}
\begin{aligned}
\|\partial_x (G * f)\|_{H^s(\mathbb{R}^2)}^2
&= \int_{\mathbb{R}^2} (1 + \xi^2+\eta^2)^s \left| \widehat{\partial_x (G*f)}(\xi,\eta) \right|^2 \dd\xi \dd\eta\\
&= \int_{\mathbb{R}^2} (1 + \xi^2+\eta^2)^s \frac{\xi^2}{(1 + \xi^2)^2} \left| \widehat{f}(\xi,\eta) \right|^2 \dd\xi \dd\eta\\
&\leq \int_{\mathbb{R}^2} (1 + \xi^2+\eta^2)^s  \left| \widehat{f}(\xi,\eta) \right|^2 \dd\xi \dd\eta = \|f\|_{H^{s}(\mathbb{R}^2)}^2.
\end{aligned}
\end{equation}
It follows that $\|\partial_x (G * f)\|_{H^s(\mathbb{R}^2)}  \leq \|f\|_{H^s(\mathbb{R}^2)}$.  \end{proof}

\begin{lem}\label{L2.2}
Let \(s > 2\),  then for any $u\in X^s(\mathbb{R}^2)$ the following inequalities hold
\begin{equation}\label{eq:23}
\begin{aligned}
\|u^2_x\|_{H^s(\mathbb{R}^2)} \leq C\|u_x\|_{H^s(\mathbb{R}^2)}^2 \leq C\|u\|_{X^s(\mathbb{R}^2)}^2,\\
 \|u^2\|_{H^s(\mathbb{R}^2)} \leq C\|u\|_{H^s(\mathbb{R}^2)}^2 \leq C\|u\|_{X^s(\mathbb{R}^2)}^2.
\end{aligned}
\end{equation}
\end{lem}

\begin{proof}
The space $H^s(\mathbb{R}^2)(s>2)$ is an algebra \cite{WMX2013}. Therefore, the conclusion follows from the inequality $$\|fg\|_{H^s(\mathbb{R}^2)} \leq \|f \|_{H^s(\mathbb{R}^2)} \|g\|_{H^s(\mathbb{R}^2)}$$ and the norm definition of the space $X^s(\mathbb{R}^2)$.
\end{proof}

\begin{lem}\label{L2.3} (Sobolev Embedding)(\cite{GLL2021})
For $s \geq 2$, the space $X^s(\mathbb{R}^2)$ is embedded into $\mathrm{Lip}(\mathbb{R}^2)$, i.e., there exists a constant $C$ such that
$\|\nabla u\|_{L^\infty(\mathbb{R}^2)} \leq C \|u\|_{X^s(\mathbb{R}^2)}$.
\end{lem}

\begin{lem}\label{C2.1} ( \cite{GLL2021}) For all $u \in X^s(\mathbb R^2)$ with $s \geq 2$, the following three inequalities hold:
\begin{align*}
&(1)\ \| u\|_{L^{\infty}}^2\leq C\| u\|_{L^{2}}^{\frac{1}{2}}\| u_{x}\|_{L^{2}}^{\frac{1}{2}}\| u_{y}\|_{L^2}^{\frac{1}{2}}
\| u_{xy}\|_{L^2}^{\frac{1}{2}};\\
&(2)\ \| u\|_{L^{\infty}}^3\leq C\| u\|_{L^{2}}\| u_{x}\|_{L^{2}}\| u_{y}\|_{L^{\infty}};\\
&(3)\ \| u\|_{L^{\infty}}\leq C\left(\| u\|_{L^{2}}+\| u_{x}\|_{L^{2}}+\| u_{y}\|_{L^{\infty}}\right).
\end{align*}
\end{lem}

\begin{lem}\label{L2.4} (\cite{BCD2011}) Let \(g\) be a smooth function on \(\mathbb R\) such that \(g(0) = 0\). If \(u \in H^s(\mathbb R^2) \cap L^{\infty}(\mathbb R^2)\) for \(s > 0\), then the composite function $g \circ u$ also belongs to the space $H^s(\mathbb{R}^2) \cap L^{\infty}(\mathbb{R}^2)$, and we have
\[ \| g \circ u \|_{H^s} \leq C(g', \|u\|_{L^{\infty}}) \|u\|_{H^{s}}.\]
Here, $C(g',\| u\|_{L^{\infty}})$ depends on the maximum value of the function $g$ and its derivative defined on the closed ball $B(0, \|u\|_{L^{\infty}})$ .
\end{lem}

\begin{lem}\label{L2.5} (\cite{BCD2011}) Let \(g\) be a smooth function on \(\mathbb R\) such that \(g(0) = 0\). If $u \in L^{\infty}(\mathbb R^{2})$, then the composite functions $g' \circ u$ and $g'' \circ u$ also belong to the space $L^{\infty}(\mathbb R^{2})$ and satisfy
\begin{align*}
 \| g^{\prime}\circ u\|_{L^{\infty}}\leq C(g^{\prime},\| u\|_{L^{\infty}})\left(1+\| u\|_{L^{\infty}}\right),\\
\| g^{\prime\prime}\circ u\|_{L^{\infty}}\leq C(g^{\prime\prime},\| u\|_{L^{\infty}})(1+\| u\|_{L^{\infty}}).
\end{align*}
Here, $C\left(g^{\prime},\| u\|_{L^{\infty}}\right)$, $C\left(g^{\prime\prime},\| u\|_{L^{\infty}}\right)$ depends on the maximum value of the function $g$ and its derivative defined on the closed ball $B(0, \|u\|_{L^{\infty}})$ .
\end{lem}



\noindent \textbf{Remark:} The constant $C$ in this paper is not necessarily equal.

\section{Blow-up necessary condition of solution}\label{S2}

The literature \cite{WKZ2022} establishes blow-up criteria for the cases $g(u)=u^3$ or $u^4$, demonstrating that a necessary condition for solutions of system \eqref{eq:3} to blow up in finite time is the divergence of the gradient integral. However, Theorem \ref{T1.1} shows that, for general infinitely differentiable functions $g(u)$, a necessary condition for finite-time blow-up of solutions to system \eqref{eq:3} is the divergence of the integral of the squared gradient.

\noindent\textbf{The proof of Theorem \ref{T1.1}}

Let the global maximal existence time of the solution $u$ be $T_{u_0}^* < \infty$. Prove the conclusion by contradiction. Suppose the integral $\int_0^{T_{u_0}^*} \|\nabla u(\tau)\|_{L^\infty(\mathbb{R}^2)}^2 \mathrm{d}\tau$ is bounded, then $\int_0^{T_{u_0}^*} \|\nabla u(\tau)\|_{L^\infty(\mathbb{R}^2)} \mathrm{d}\tau$ is also bounded. Let $M>0$ large enough and determined later such that $\sup\limits_{\tau\in[0,T^*_{u_0}]}\| u(\tau)\|_{L^\infty}<M.$

First, we derive the energy estimate for \(X^s\): for \(s \geq 2\), there exists \(T > 0\) such that for any \(t \in [0, T]\), we have
\begin{equation}\label{eq:25}
\sup_{\tau\in[0,t]}\| u(\tau)\|_{X^{s}}\leq 2\| u_{0}\|_{X^{s}}.
\end{equation}

\textbf{Step 1}: $L^2-$estimate. From $u_{y}=v_{x}$, we get $\partial_{x}^{-1}u_{y}=v$. Applying $\partial_{x}^{-1}$ to  \eqref{GCHKP}, taking the $ L^{2}$ inner product with $u$ and integrating by parts, we have
\begin{equation}\label{eq:26}
\begin{split}
\frac{1}{2}\frac{\mathrm{d}}{\mathrm{d}t}\left(\| u\|_{L^{2}}^{2}+\| u_{x}\|_{L^{2}}^{2}\right)
=-\int_{\mathbb{R}^{2}}\left(\frac{g(u)}{2}\right)_{x} u \mathrm{d}x \mathrm{d}y
+\gamma\int_{\mathbb{R}^{2}}(2u_{x}u_{xx}u+u^{2}u_{xxx})\mathrm{d}x \mathrm{d}y.
\end{split}
\end{equation}
Using integration by parts, we obtain $\int_{\mathbb{R}^{2}}(2uu_{x}u_{xx}+u^{2}u_{xxx})\mathrm{d}x \mathrm{d}y=0$ and
\(-\int_{\mathbb{R}^{2}}\left(\frac{g(u)}{2}\right)_{x}u\mathrm{d}x\mathrm{d}y=0,\) Due to the smoothness of $g(u)$.
Thus,
\begin{equation}\label{eq:27}
\frac{\mathrm{d}}{\mathrm{d}t}\left(\| u\|_{L^{2}}^{2}+\| u_{x}\|_{L^{2}}^{2}\right)=0.
\end{equation}

\textbf{Step 2}: $X^s-$estimate. Applying the operators $\partial_x^{-1}$ and $\Lambda^s$  to equation \eqref{GCHKP} respectively, one yields:
\begin{equation}\label{eq:28}
\partial_{t}\partial_{x}^{-1}\Lambda^{s}u-\partial_ {t}\Lambda^{s}u_{x}+\frac{1}{2}\Lambda^{s}g(u)-\gamma\Lambda^{s}\left(uu_{x}\right)_{x}+\frac{\gamma}{2}\Lambda^{s}\left(u_{x}^{2}\right)
+\Lambda^{s}\partial_{y}\partial_{x}^{-1}v=0.
\end{equation}
Taking the $L^2$ inner product of equation \eqref{eq:28} with $\Lambda^{s}\partial_{x}^{-1}u$ over the whole space, we obtain
\begin{equation}\label{eq:29}
\begin{split}
\frac{1}{2}\frac{\mathrm{d}}{\mathrm{d}t}(\| \partial_{x}^{-1}\Lambda^{s}u\|_{L^{2}}^{2}+\| \Lambda^{s}u\|_{L^{2}}^{2})=\sum_{i=1}^{4}A_{i},
\end{split}
\end{equation}
where
$$A_{1}=-\int_{\mathbb{R}^{2}}\Lambda^{s}\partial_{x}^{-1}v_{y}\Lambda^{s}\partial_{x}^{-1}u\mathrm{d}x\mathrm{d}y,
\quad A_{2}=-\int_{\mathbb{R}^{2}}\frac{1}{2}\Lambda^{s}g(u)\Lambda^{s}\partial_{x}^{-1}u\mathrm{d}x\mathrm{d}y,$$
$$A_{3}=\int_{\mathbb{R}^{2}}\gamma \Lambda^{s}\left(uu_{x}\right)_{x}\Lambda^{s}\partial_{x}^{-1}u\mathrm{d}x\mathrm{d}y,
\quad A_{4}=-\int_{\mathbb{R}^{2}}\frac{\gamma}{2}\Lambda^{s}\left(u_{x}^{2}\right)\Lambda^{s}\partial_{x}^{-1}u\mathrm{d}x\mathrm{d}y.$$
Clearly, $A_1=0$. By Lemma \ref{L2.4},
$$A_2\leq C(g^{\prime},\| u\|_{L^{\infty}})(\| u\|_{H^{s}}^{2}+\|\partial_{x}^{-1}u\|_{H^{s}}^{2}).$$
According to the product estimate, we obtain
$$A_3\leq C\gamma \| u\|_{L^{\infty}}(\| u\|_{H^{s}}^{2}+\| u_{x}\|_{H^{s}}^{2}),\quad
A_4\leq C\gamma \| u_{x}\|_{L^{\infty}}(\| u_{x}\|_{H^{s}}^{2}+\|\partial_{x}^{-1}u\|_{H^{s}}^{2}).$$
Substituting the estimates for $A_1$ to $A_4$ into equation \eqref{eq:29} yields that for $s \geq 2$, we have
\begin{equation}\label{eq:30}
\begin{split}
\frac{\mathrm{d}}{\mathrm{d}t}(\| \partial_{x}^{-1}u\|_{H^{s}}^{2}+\| u\|_{H^{s}}^{2})
&\leq C(\gamma,g^{\prime},\|u\|_{L^{\infty}})(1+\|u\|_{L^{\infty}}+\| u_{x}\|_{L^{\infty}})\\
&\quad \times\big(\| u\|_{H^{s}}^{2}+\| u_{x}\|_{H^{s}}^{2}+\| \partial_{x}^{-1}u\|_{H^{s}}^{2}\big).
\end{split}
\end{equation}
Combining this with the definition of the $X^s$ norm, we can derive that
\begin{equation}\label{eq:31}
\frac{\mathrm{d}}{\mathrm{d}t}\left(\| \partial_{x}^{-1}u\|_{H^{s}}^{2}
+\| u\|_{H^{s}}^{2}\right)\leq  C(\gamma,g^{\prime},\|u\|_{L^{\infty}})(1+\|u\|_{L^{\infty}}
+\| u_{x}\|_{L^{\infty}})\| u\|_{X^{s}}^{2}\,\, (s\geq2).
\end{equation}

To close the inequality \eqref{eq:30}, we will estimate $\|\Lambda^{s}u_{x}\|_{L^{2}}^{2}$  by considering the cases  $s>2$ and $s=2$.

\textbf{Case 1:}  $s > 2$. Applying the operator $\Lambda^s$ to equation \eqref{GCHKP} and taking the $L^2$ inner product with $\Lambda^{s} u$ over the entire space, we find
\begin{equation}\label{eq:32}
\begin{split}
&\qquad\frac {1}{2}\frac{\mathrm{d}}{\mathrm{d}t}(\|\Lambda^{s}u\|_{L^{2}}^{2}+\|\Lambda^{s}u_{x}\|_{L^{2}}^{2})\\
&=-\int_{\mathbb{R}^{2}}
\left(\frac {1}{2}\Lambda^{s}g(u)-\frac {\gamma}{2}\Lambda^{s}\left(u_{x}^{2}\right)\right)_{x}\Lambda^{s}u\mathrm{d}x\mathrm{d}y\\
&\quad+\gamma\int_{\mathbb{R}^{2}}\Lambda^{s}\left(uu_{xx}\right)_{x}\Lambda^{s}u\mathrm{d}x\mathrm{d}y
-\int_{\mathbb{R}^{2}}\Lambda^{s}v_{y}\Lambda^{s}u
\mathrm{d}x\mathrm{d}y\\
&=:\sum_{i=1}^{3}B_{i}.
\end{split}
\end{equation}
By the product estimate and Lemma \ref{L2.4}, we obtain
$$B_1 \leq C(\gamma)\big[C(g^{\prime},\|u\|_{L^{\infty}})+\|u_{x}\|_{L^{\infty}}\big](\|u\|_{H^{s}}^2+\| u_{x}\|_{H^{s}}^2).$$
Applying the commutator estimate, it is easy to follow that
$$| B_{2}| \leq C(\gamma)(\|\nabla u\|_{L^{\infty}}+\|u_{xx}\|_{L^{\infty}})(\| u\|_{H^{s}}^{2}+\| u_{x}\|_{H^{s}}^{2}).$$
Clearly, $B_3=0$. Inserting these estimates for $B_1$, $B_2$ and $B_3$ into \eqref{eq:32}, one gives
\begin{equation}\label{eq:33}
\begin{split}
\frac{\mathrm{d}}{\mathrm{d}t}\big(\| u\|_{H^{s}}^{2}+\| u_{x}\|_{H^{s}}^{2}\big)
\leq C(\gamma,g^{\prime},\|u\|_{L^{\infty}})(1+\|\nabla u\|_{L^{\infty}}+\| u_{xx}\|_{L^{\infty}})\|u\|_{X^{s}}^2.
\end{split}
\end{equation}
Combining equations \eqref{eq:31} and \eqref{eq:33} then  it is easy to yield
\begin{equation}\label{eq:34}
\begin{split}
\frac{\mathrm{d}}{\mathrm{d}t}\left(2\| u\|_{H^{s}}^{2}+\| u_{x}\|_{H^{s}}^{2}+\| \partial_{x}^{-1}u\|_{H^{s}}^{2}\right)
&\leq C(\gamma, g',\|u\|_{L^\infty})(2+\|u\|_{L^\infty}+\|u_x\|_{L^\infty}\\
&\quad+\|\nabla u\|_{L^\infty}+\|u_{xx}\|_{L^\infty})\|u\|_{X^s}^2.
\end{split}
\end{equation}

\textbf{Case 2:} $s=2$. The embedding $H^{1}\hookrightarrow L^{\infty}$ fails, so $\| u_{xx}\|_{L^{\infty}}$ in equation \eqref{eq:33} can no longer be bounded using the embedding inequality. Thus, a new approach is required to estimate the higher-order derivatives of $ u $ for $ s = 2$.  Following the arguments in \cite{GLL2021} and Lemma \ref{C2.1}, we have
\begin{align}
\|\nabla u\|_{L^{\infty}}&\leq C(\|\nabla u\|_{L^{2}}+\|\nabla u_{x}\|_{L^{2}}+\|\nabla u_{y}\|_{L^{\infty}}),\label{eq:35}\\
\| u_{xx}\|_{L^{\infty}}&\leq C(\| u_{xx}\|_{L^{2}}+\| u_{xxx}\|_{L^{2}}+\| u_{xxy}\|_{L^{\infty}}).\label{eq:36}
\end{align}
To control these terms on the right-hand side of the above two inequalities, namely $\|\nabla u\|_{L^{2}}$, \(\|\nabla u_{x}\|_{L^{2}}\) and \(\| u_{xxx}\|_{L^{2}}\), we apply the operators \(\nabla\), \(\nabla\partial_{x}\), and \(\nabla\partial_{x}^{2}\) to the first equation of system \eqref{eq:3}. This yields
\begin{align}
\partial_{t}\nabla u&+\gamma u\partial_{x}\nabla u+\gamma u_{x}\nabla u
+G_{x}\ast\big(\frac{1}{2}g^{\prime}(u)\nabla u
-\gamma u\nabla u+\gamma u_{x}\nabla u_{x}\big)+G\ast\nabla v_{y}=0,\label{eq:37}\\
\partial_{t}\nabla u_{x}&+\gamma u\partial_{x}\nabla u_{x}+\gamma u_{x}\nabla u_{x}+\gamma u_{xx}\nabla u-\frac{1}{2}g^{\prime}(u)\nabla u+\gamma u\nabla u\nonumber\\
&\quad+G\ast\left(\frac{1}{2}g^{\prime}(u)\nabla u-\gamma u\nabla u+\gamma u_{x}\nabla u_{x}\right)+G_{x}\ast\nabla v_{y}=0,\label{eq:38}\\
\partial_{t}\nabla u_{xx}&+\gamma u\partial_{x}\nabla u_{xx}+2\gamma u_{x}\nabla u_{xx}+2\gamma u_{xx}\nabla u_{x}+\gamma u_{xxx}\nabla u\nonumber\\
&\quad-\frac{1}{2}\left[g^{\prime\prime}(u)u_{x}\nabla u+g^{\prime}(u)\nabla u_{x}\right]+\gamma u_{x}\nabla u+\gamma u\nabla u_{x}\nonumber\\
&\quad+G_{x}\ast\left(\frac{1}{2}g^{\prime}(u)\nabla u-\gamma u\nabla u+\gamma u_{x}\nabla u_{x}\right)+G_{xx}\ast\nabla v_{y}=0.\label{eq:39}
\end{align}
 Taking the $L^2$ inner product of \eqref{eq:37}-\eqref{eq:39} with $\nabla u, \nabla \partial_x u, \nabla \partial_x^2 u$ respectively and summing the results, we obtain
\begin{equation}\label{eq:40}
\frac{1}{2}\frac{\mathrm{d}}{\mathrm{d}t}\left(\|\nabla u\|_{L^{2}}^{2}+\|\nabla u_{x}\|_{L^{2}}^{2}\right)=I_{1}+I_{3}+I_{4}+I_{6},
\end{equation}
\begin{equation}\label{eq:41}
\frac{1}{2}\frac{\mathrm{d}}{\mathrm{d}t}\left(\|\nabla u\|_{L^{2}}^{2}+\|\nabla u_{x}\|_{L^{2}}^{2}+\|\nabla u_{xx}\|_{L^{2}}^{2}\right)=\sum_{j=1}^{7}I_{j},
\end{equation}
where
\begin{align*}
&I_{1}=-\gamma\int_{\mathbb{R}^{2}}(u\partial_{x}\nabla u\cdot\nabla u+u\partial_{x}\nabla u_{x}\cdot\nabla u_{x})\mathrm{d}x\mathrm{d}y,\\
&I_{2}=-\gamma\int_{\mathbb{R}^{2}}u\partial_{x}\nabla u_{xx}\cdot\nabla u_{xx}\mathrm{d}x\mathrm{d}y,\\
&I_{3}=-\int_{\mathbb{R}^{2}}\left[\gamma u_{x}\nabla u+G_{x}\ast\left(\frac{1}{2}g^{\prime}(u)\nabla u-\gamma u\nabla u
+\gamma u_{x}\nabla u_{x}\right)\right]\cdot\nabla u\mathrm{d}x\mathrm{d}y,\\
&I_{4}=-\int_{\mathbb{R}^{2}}[\gamma u_{x}\nabla u_{x}+\gamma u_{xx}\nabla u-\frac{1}{2}g^{\prime}(u)\nabla u+\gamma u\nabla u\\
&\qquad\qquad\quad+G\ast(\frac{1}{2}g^{\prime}(u)\nabla u-\gamma u\nabla u+\gamma u_{x}\nabla u_{x})]\cdot\nabla u_{x}\mathrm{d}x\mathrm{d}y,\\
&I_{5}=-\int_{\mathbb{R}^{2}}[2\gamma u_{x}\nabla u_{xx}+2\gamma u_{xx}\nabla u_{x}+\gamma u_{xxx}\nabla u\\
&\qquad\quad\qquad-\frac{1}{2}(g^{\prime\prime}(u)u_{x}\nabla u+g^{\prime}(u)\nabla u_{x})+\gamma u_{x}\nabla u+\gamma u\nabla u_{x}\\
&\qquad\quad\qquad+G_{x}\ast(\frac{1}{2}g^{\prime}(u)\nabla u-\gamma u\nabla u
+\gamma u_{x}\nabla u_{x})]\cdot\nabla u_{xx}\mathrm{d}x\mathrm{d}y,\\
&I_{6}=-\int_{\mathbb{R}^{2}}G\ast\nabla v_{y}\cdot\nabla udxdy-\int_{\mathbb{R}^{2}}G_{x}\ast\nabla v_{y}\cdot\nabla u_{x}\mathrm{d}x\mathrm{d}y,\\
&I_{7}=-\int_{\mathbb{R}^{2}}G_{xx}\ast\nabla v_{y}\cdot\nabla u_{xx}\mathrm{d}x\mathrm{d}y.
\end{align*}
One directly calculates to show that
$$I_6=I_7=0,\quad |I_{1}|\leq C(\gamma)\| \nabla u\|_{L^{\infty}}(\|\nabla u\|_{L^{2}}^{2}+\|\nabla u_{x}\|_{L^{2}}^{2}),
\quad |I_{2}|\leq C(\gamma)\|\nabla u\|_{L^{\infty}}\|\nabla u_{xx}\|_{L^{2}}^{2}.$$
By Lemma \ref{L2.5}, we obtain
\begin{align*}
| I_{3}|&
\leq C(\gamma,g^{\prime},\| u\|_{L^{\infty}})\Bigl((1+\|u\|_{L^{\infty}})\|\nabla u\|_{L^{2}}^{2}
+\|\nabla u\|_{L^{\infty}}\big(\|\nabla u\|_{L^{2}}^{2}
+\|\nabla u_{x}\|_{L^{2}}^{2}\big)\Bigr),\\
| I_{4}|&
\leq C(\gamma,g^{\prime},\| u\|_{L^{\infty}})\Bigl((1+\|u\|_{L^{\infty}})(\|\nabla u\|_{L^{2}}^{2}+\|\nabla u_{x}\|_{L^{2}}^{2})
+\|\nabla u\|_{L^{\infty}}\|\nabla u_{x}\|_{L^{2}}^{2}\Bigr).
\end{align*}
By the H\"{o}lder inequality, Lemma \ref{L2.5}, and
\begin{equation}\label{eq:42}
\begin{split}
2\gamma \| u_{xx}\|_{L_{x}^{2}L_{y}^{\infty}}\|\nabla u_{x}\|_{L_{x}^{\infty}L_{y}^{2}}
\leq&C(\gamma)\| u_{xx}\|_{L^{2}}^{\frac{1}{2}}\| u_{xxy}\|_{L^{2}}^{\frac{1}{2}}\| \nabla u_{x}\|_{L^{2}}^{\frac{1}{2}}\| \nabla u_{xx}\|_{L^{2}}^{\frac{1}{2}}\\
\leq&C(\gamma)\| \nabla u_{x}\|_{L^{2}}\| \nabla u_{xx}\|_{L^{2}},
\end{split}
\end{equation}
we conclude that
\begin{equation}\label{eq:43}
\begin{split}
| I_{5}|
&\leq C(\gamma,g^{\prime},g^{\prime\prime},\|u\|_{L^{\infty}})\Bigl((1+\|u\|_{L^{\infty}})(1+\|\nabla u\|_{L^{\infty}})+\|\nabla u_{x}\|_{L^{2}}\Bigr)\\
&\quad\times(\|\nabla u\|_{L^2}^2+\|\nabla u_{x}\|_{L^2}^2+\|\nabla u_{xx}\|_{L^2}^2).
\end{split}
\end{equation}
Substituting the various estimates from $I_1$ to $I_7$ into \eqref{eq:40} and \eqref{eq:41}, it  implies
\begin{align}
&\quad\frac{\mathrm{d}}{\mathrm{d}t}(\|\nabla u\|_{L^{2}}^{2}+\|\nabla u_{x}\|_{L^{2}}^{2})\nonumber\\
&\leq C(\gamma,g^{\prime},\|u\|_{L^{\infty}})\big(1+\|\nabla u\|_{L^{\infty}}
+\|u\|_{L^{\infty}}\big)\big(\|\nabla u\|_{L^2}^2+\|\nabla u_{x}\|_{L^2}^2\big),\label{eq:44}\\
&\quad\frac{\mathrm{d}}{\mathrm{d}t}(\|\nabla u\|_{L^{2}}^{2}+\|\nabla u_{x}\|_{L^{2}}^{2}+\|\nabla u_{xx}\|_{L^{2}}^{2})\nonumber\\
&\leq C(\gamma,g^{\prime},g^{\prime\prime},\|u\|_{L^{\infty}})\Bigl(\|\nabla u_{x}\|_{L^{2}}
+(1+\|u\|_{L^{\infty}})(1+\|\nabla u\|_{L^{\infty}})\Bigr)\times(\|\nabla u\|_{L^2}^2+\|\nabla u_{x}\|_{L^2}^2+\|\nabla u_{xx}\|_{L^2}^2).\label{eq:45}
\end{align}
Now we estimate $\|\nabla u_{y}\|_{L^{\infty}}$ and $\| u_{xxy}\|_{L^{\infty}}$ in the inequalities \eqref{eq:35} to \eqref{eq:36}. Applying the operators $\partial_{y}^{2}$ and $\partial_{x}\partial_{y}^{2}$ to the first equation of system \eqref{eq:3}, respectively, it yields
\begin{align}
\partial_{t}u_{yy}&+\gamma u\partial_{x}u_{yy}+2\gamma u_{y}u_{xy}+\gamma u_{x}u_{yy}+G_{x}\ast(\frac{1}{2}g^{\prime\prime}(u)u_{y}^{2}+
\frac{1}{2}g^{\prime}(u)u_{yy}\nonumber\\
&\quad+\gamma u_{xy}^{2}+\gamma u_{x}u_{xyy}-\gamma u_{y}^{2}-\gamma uu_{yy})+G\ast v_{yyy}=0,\label{eq:46}\\
\partial_{t}u_{xyy}&+\gamma u\partial_{x}u_{xyy}+\gamma u_{x}u_{xyy}+\gamma u_{xy}^{2}+2\gamma u_{y}u_{xxy}+\gamma u_{xx}u_{yy}
-\frac{1}{2}g^{\prime\prime}(u)u_{y}^{2}\nonumber\\
&\quad-\frac{1}{2}g^{\prime}(u)u_{yy}+\gamma u_{y}^{2}+\gamma uu_{yy}+G\ast(\frac{1}{2}g^{\prime\prime}(u)u_{y}^{2}+\frac{1}{2}g^{\prime}(u)u_{yy}+\gamma u_{xy}^{2}\nonumber\\
&\quad+\gamma u_{x}u_{xyy}-\gamma u_{y}^{2}-\gamma uu_{yy})+G_{x}\ast v_{yyy}=0.\label{eq:47}
\end{align}
Taking the $L^2$ inner products of \eqref{eq:46} and \eqref{eq:47} with $u_{yy}$ and $u_{xyy}$, respectively, over the whole space, and adding the resulting equations, we obtain
\begin{equation}\label{eq:48}
\frac{1}{2}\frac{\mathrm{d}}{\mathrm{d}t}(\| u_{yy}\|_{L^{2}}^{2}+\| u_{xyy}\|_{L^{2}}^{2})=\sum_{i=1}^{4}J_{i},
\end{equation}
where
\begin{align*}
&J_{1}=-\gamma\int_{\mathbb{R}^{2}}(uu_{xyy}u_{yy}+uu_{xxyy}u_{xyy})\mathrm{d}x\mathrm{d}y,\\
&J_{2}=-\int_{\mathbb{R}^{2}}\big\{(2\gamma u_{y}u_{xy}+\gamma u_{x}u_{yy})u_{yy}
+[\gamma u_{x}u_{xyy}+\gamma u_{xy}^{2}+2\gamma u_{y}u_{xxy}+\gamma u_{xx}u_{yy}\\
&\qquad\qquad\quad-\frac{1}{2}g^{\prime\prime}(u)u_{y}^{2}-\frac{1}{2}g^{\prime}(u)u_{yy}+\gamma u_{y}^{2}+\gamma uu_{yy}]u_{xyy}\big\}\mathrm{d}x\mathrm{d}y,\\
&J_{3}=-\int_{\mathbb{R}^{2}}\big[G_{x}\ast(\frac{1}{2}g^{\prime\prime}(u)u_{y}^{2}+\frac{1}{2}g^{\prime}(u)u_{yy}+\gamma u_{xy}^{2}+\gamma u_{x}u_{xyy}-\gamma u_{y}^{2}-\gamma uu_{yy})u_{yy}\\
&\qquad+G\ast(\frac{1}{2}g^{\prime\prime}(u)u_{y}^{2}+\frac{1}{2}g^{\prime}(u)u_{yy}+\gamma u_{xy}^{2}+\gamma u_{x}u_{xyy}-\gamma u_{y}^{2}-\gamma uu_{yy})u_{xyy}\big]\mathrm{d}x\mathrm{d}y,\\
&J_{4}=-\int_{\mathbb{R}^{2}}(G\ast v_{yyy}u_{yy}+G_{x}\ast v_{yyy}u_{xyy})\mathrm{d}x\mathrm{d}y.
\end{align*}
Direct calculations show that
$$J_3=J_4=0,\quad |J_{1}|\leq C(\gamma)\| \nabla u\|_{L^{\infty}}(\| u_{yy}\|_{L^{2}}^{2}+\| u_{xyy}\|_{L^{2}}^{2}).$$
Combining the H\"{o}lder's inequality with the Young's inequality, we obtain
\begin{equation}\label{eq:49}
\begin{split}
|\int_{\mathbb{R}^{2}}\gamma u_{xy}^{2}u_{xyy}\mathrm{d}x\mathrm{d}y|
&\leq C(\gamma)\| u_{xy}\|_{L_{x}^{2}L_{y}^{\infty}}\| u_{xy}\|_{L_{x}^{\infty}L_{y}^{2}}\|u_{xyy}\|_{L^{2}}\\
&\leq C(\gamma)\| u_{xy}\|_{L^{2}}^{\frac{1}{2}}\| u_{xyy}\|_{L^{2}}^{\frac{1}{2}}\| u_{xy}\|_{L^{2}}^{\frac{1}{2}}
\| u_{xxy}\|_{L^{2}}^{\frac{1}{2}}\|u_{xyy}\|_{L^{2}}\\
&\leq C(\gamma)\| u_{xy}\|_{L^{2}}\| u_{xyy}\|_{L^{2}}^{\frac{3}{2}}
\| u_{xxy}\|_{L^{2}}^{\frac{1}{2}}\\
&\leq C(\gamma)\|\nabla u_{x}\|_{L^{2}}(\| u_{xyy}\|_{L^{2}}^{2}+
\| u_{xxy}\|_{L^{2}}^{2}),
\end{split}
\end{equation}
\begin{equation}\label{eq:50}
\begin{split}
|\int_{\mathbb{R}^{2}}\gamma u_{xx}u_{yy}u_{xyy}\mathrm{d}x\mathrm{d}y| \leq C(\gamma)\|(\nabla u_{x},\nabla u_{xx})\|_{L^{2}}\|(u_{yy},u_{xyy})\|_{L^{2}}^{2}.
\end{split}
\end{equation}
Using the same estimation strategy as in \eqref{eq:49} together with Lemma \ref{L2.5}, we conclude that
\begin{align}
\left|\int_{\mathbb{R}^{2}}\frac{1}{2}g^{\prime\prime}(u)u_{y}^{2}u_{xyy}\mathrm{d}x\mathrm{d}y\right|  &\leq C\| g^{\prime\prime}(u)\|_{L^{\infty}}\| u_{y}\|_{L^{2}}\| u_{yy}\|_{L^{2}}^{\frac{1}{2}}\| u_{xy}\|_{L^{2}}^{\frac{1}{2}}\| u_{xyy}\|_{L^{2}}\nonumber\\
&\leq C(g^{\prime\prime},\| u\|_{L^{\infty}})(1+\| u\|_{L^{\infty}})\|\nabla u\|_{L^{2}}\| (u_{yy},\nabla u_{x},u_{xyy})\|_{L^{2}}^{2},\label{eq:51}\\
\left|\int_{\mathbb{R}^{2}}\frac{1}{2}g^{\prime}(u)u_{yy}u_{xyy}\mathrm{d}x\mathrm{d}y\right|
&\leq C(g^{\prime},\| u\|_{L^{\infty}})(1+\| u\|_{L^{\infty}})(\| u_{yy}\|_{L^{2}}^{2}+\| u_{xyy}\|_{L^{2}}^{2}).\label{eq:52}
\end{align}
Thus we have
\begin{align*}
|J_{2}|\leq C(\gamma,g^{\prime},g^{\prime\prime},\| u\|_{L^{\infty}})&\Big\{\Bigl(\big(1+\| u\|_{L^{\infty}}\big)\|\nabla u\|_{L^2}+1+\|(\nabla u_{x},\nabla u_{xx})\|_{L^2}+\| \nabla u\|_{L^{\infty}}+\| u\|_{L^{\infty}}\Bigr)\big(\| u_{yy}\|_{L^2}^2+\| u_{xyy}\|_{L^2}^2\big)\\
&+(\| \nabla u_{x}\|_{L^{2}}+\| \nabla u\|_{L^{\infty}})
\| \nabla u_{xx}\|_{L^{2}}^2+(1+ \| u\|_{L^{\infty}})\|\nabla u\|_{L^2}\| \nabla u_{x}\|_{L^{2}}^2\Big\}.
\end{align*}
Substituting the all estimates into equation \eqref{eq:48} gives
\begin{equation}\label{eq:53}
\begin{split}
&\frac{\mathrm{d}}{\mathrm{d}t}(\| u_{yy}\|_{L^{2}}^{2}+\| u_{xyy}\|_{L^{2}}^{2})\leq C(\gamma,g^{\prime},g^{\prime\prime},\| u\|_{L^{\infty}})\big\{\big[\big(1+\| u\|_{L^{\infty}}\big)\|\nabla u\|_{L^2}+\| (\nabla u,u)\|_{L^{\infty}}\\
&\quad+\|(\nabla u_{x},\nabla u_{xx})\|_{L^2}+1\big]\big(\| u_{yy}\|_{L^2}^2
+\| u_{xyy}\|_{L^2}^2\big)+(1+ \| u\|_{L^{\infty}})\|\nabla u\|_{L^2}\| \nabla u_{x}\|_{L^{2}}^2\\
&\quad
+(\| \nabla u_{x}\|_{L^{2}}+\| \nabla u\|_{L^{\infty}})\| \nabla u_{xx}\|_{L^{2}}^2\big\}.
\end{split}
\end{equation}
From the systems of equations \eqref{eq:27}, \eqref{eq:30}, \eqref{eq:45}, \eqref{eq:53}, we can derive
\begin{equation}\label{eq:54}
\begin{split}
\frac{\mathrm{d}}{\mathrm{d}t}(\| (u,u_{x})\|_{H^{1}}^{2}&+\|(\nabla u_{xx},u_{yy},u_{xyy})\|_{L^{2}}^{2}
+\| (\partial_{x}^{-1}u,u)\|_{H^{2}}^{2})\\
&\leq C(\gamma,g^{\prime},g^{\prime\prime},\|u\|_{L^{\infty}})\Bigl(\big(\|\nabla u\|_{L^{\infty}}+\|\nabla u\|_{L^{2}}+1\big)\big(1+\|u\|_{L^{\infty}}\big)\\
&\qquad\qquad\qquad\quad+\|\nabla u_{x}\|_{L^{2}}+\|\nabla u_{xx}\|_{L^{2}}\Bigr)
\|u\|_{X^{2}}^2.
\end{split}
\end{equation}
On the one hand, when $s>2$, we have
\begin{equation}\label{eq:55}
\frac{1}{2}\| u\|_{X^{s}}^{2}\leq2\| u\|_{H^{s}}^{2}+\| u_{x}\|_{H^{s}}^{2}+\| \partial_{x}^{-1}u\|_{H^{s}}^{2}\leq2\| u\|_{X^{s}}^{2}.
\end{equation}
Then combining equation \eqref{eq:55} with equation \eqref{eq:34}, it implies
\begin{equation}\label{eq:56}
\frac{\mathrm{d}}{\mathrm{d}t}(\| u\|_{X^{s}}^{2})\leq C(\gamma,g^{\prime},M)(\|u\|_{X^{s}}+1)\|u\|_{X^{s}}^2.
\end{equation}
According to Gr\"{o}nwall's inequality, we obtain
\begin{equation}\label{eq:57}
\sup_{\tau\in[0,t]}\| u\|_{X^{s}}^{2}\leq\| u_{0}\|_{X^{s}}^{2}e^{C(\gamma,g^{\prime},M)\int_{0}^{t}(\|u\|_{X^{s}}+1)(\tau)\mathrm{d}\tau}.
\end{equation}
Let \(T > 0\) satisfy the inequality \(2C(\gamma, g', M) \left( \|u_0\|_{X^s} + 1 \right) T \leq 1\). Then, by the bootstrap method, we have
\begin{equation}\label{eq:58}
\sup_{\tau\in[0,t]}\| u(\tau)\|_{X^{s}}\leq 2\| u_{0}\|_{X^{s}},\ \forall\ t\in [0,T].
\end{equation}
For  $T_{u_0}^*$, choose a constant $M>0$, such that $M>C(T_{u_0})\|u_0\|_{X^s}\geq C(T_{u_0})\sup_{\tau\in[0,T_{u_0}]}\| u(\tau)\|_{X^{s}} \geq \|u\|_{L^{\infty}}$.

On the other hand, when $s=2$, we have
\begin{equation}\label{eq:59}
\begin{split}
\frac{1}{2}\| u\|_{X^{s}}^{2}&\leq\| (u,u_{x})\|_{H^{1}}^{2}+\|(\nabla u_{xx},u_{yy},u_{xyy})\|_{L^{2}}^{2}
+\| (\partial_{x}^{-1}u,u)\|_{H^{2}}^{2}\leq2\| u\|_{X^{s}}^{2}.
\end{split}
\end{equation}
Thus, combining the inequalities \eqref{eq:59} and \eqref{eq:54}, one readily yields
\begin{equation}\label{eq:60}
\begin{aligned}
\frac{\mathrm{d}}{\mathrm{d}t}(\| u\|_{X^{s}}^{2})\leq C(\gamma,g^{\prime},g^{\prime\prime},M)\big[&(\|\nabla u\|_{L^{\infty}}+\|\nabla u\|_{L^{2}}+1)(1+\|u\|_{L^{\infty}})\\
&+\|\nabla u_{x}\|_{L^{2}}+\|\nabla u_{xx}\|_{L^{2}}\big]\| u\|_{X^{s}}^{2}.
\end{aligned}
\end{equation}
Applying  Gr\"{o}nwall's inequality, we obtain
\begin{equation}\label{eq:61}
\begin{split}
\sup_{\tau\in[0,t]}\| u\|_{X^{s}}^{2}&\leq\| u_{0}\|_{X^{s}}^{2}\exp\bigg\{C(\gamma,g^{\prime},g^{\prime\prime},M)\int_{0}^{t}
\bigg[\|(\nabla u_{x},\nabla u_{xx})\|_{L^{2}}\\
&\quad\quad\qquad +(\|\nabla u\|_{L^{\infty}}+\|\nabla u\|_{L^{2}}+1)(1+\|u\|_{L^{\infty}})\bigg](\tau)\mathrm{d}\tau\bigg\}.
\end{split}
\end{equation}
Choose $T > 0$ such that
\begin{equation*}
\begin{split}
2C(\gamma,g^{\prime},g^{\prime\prime},M)\int_{0}^{T}\big[&(\|\nabla u\|_{L^{\infty}}
+\|\nabla u\|_{L^{2}}+1)(1+\|u\|_{L^{\infty}})\\
&+\|\nabla u_{x}\|_{L^{2}}+\|\nabla u_{xx}\|_{L^{2}}\big](\tau)\mathrm{d}\tau\leq1.
\end{split}
\end{equation*}
Thus, by the bootstrap method, the inequality \eqref{eq:58} remains valid. This completes the proof of the energy estimate.

We now proceed to prove the blow-up criterion. From Lemma \ref{C2.1} and equation \eqref{eq:27}, it follows that for $s \geq 2$,
\begin{equation}\label{eq:62}
\begin{aligned}
\|u\|_{L^\infty} &\leq C\left(\|u\|_{L^2} + \|u_x\|_{L^2} + \|u_y\|_{L^\infty}\right) \\
&= C\left(\|u_0\|_{L^2} + \|u_{0x}\|_{L^2} + \|u_y\|_{L^\infty}\right) \\
&\leq C\left(\|u_0\|_{L^2},\|u_{0x}\|_{L^2}\right)\|\nabla u\|_{L^\infty}.
\end{aligned}
\end{equation}
In the following, we divide the discussion of the blow-up criterion into two cases: $s > 2$ and $s = 2$.

\textbf{Case 1:} $s>2$. From $H^s(\mathbb{R}^2) \hookrightarrow L^\infty(\mathbb{R}^2)$ and \eqref{eq:58}, it follows that when $s > 2$, for any $t \in [0,T^*_{u_0})$, we have
\begin{align}\label{eq:63}
\|u_{xx}\|_{L^\infty} \leq \|u\|_{X^s} \leq 2\|u_0\|_{X^s}.
\end{align}
Substituting estimates \eqref{eq:62} and \eqref{eq:63} into the inequality \eqref{eq:34} together with the definition of the $X^s$ norm, we obtain
\begin{align}\label{eq:64}
\frac{\mathrm{d}}{\mathrm{d}t}\|u\|_{X^s}^2
\leq C(\gamma, g',M)\left(\|u_0\|_{L^2} + \|u_{0x}\|_{L^2} + 2 + 2\|\nabla u\|_{L^\infty}+2\|u_0\|_{X^s}\right)\|u\|_{X^s}^2.
\end{align}
By the Gr\"{o}nwall's inequality, we obtain
\begin{equation}\label{eq:65}
\begin{aligned}
\|u\|_{X^s}^2 &\leq \|u_0\|_{X^s}^2 \exp\Bigg\{ C(\gamma, g', M) \\
&\quad \times \left[ ( \|u_0\|_{L^2} + \|u_{0x}\|_{L^2} + 2 + 2\|u_0\|_{X^s} ) t + 2 \int_0^t \|\nabla u\|_{L^\infty} \mathrm{d}\tau \right] \Bigg\}.
\end{aligned}
\end{equation}
Since $\int_0^{T_{u_0}^*} \|\nabla u(\tau)\|_{L^\infty(\mathbb{R}^2)} \mathrm{d}\tau$ is bounded, take
\begin{equation}\label{eq:66}
\begin{aligned}
M_4(T_{u_0}^*)=& \|u_0\|_{X^s}^2 \exp\Bigg\{ C(\gamma, g',M)\\
& \times \left[\left(\|u_0\|_{L^2} + \|u_{0x}\|_{L^2} + 2 + 2\|u_0\|_{X^s}\right)T_{u_0}^* + 2\int_0^{T_{u_0}^*} \|\nabla u\|_{L^\infty} \mathrm{d}\tau\right] \Bigg\}.
\end{aligned}
\end{equation}
Then, for all \(t \in [0, T_{u_0}^*)\), we have
\begin{equation}\label{eq:67}
\|u\|_{X^s}^2 \leq M_4(T_{u_0}^*).
\end{equation}
Then the norm of $u$ in the $X^s$ space is uniformly bounded before the global maximum existence time $T_{u_0}^*$.

Choose a sufficiently small $T_1$ such that $T_{u_0}^* - T_1/2 \geq 0$. Let $t_0 = T_{u_0}^* - T_1/2$, where $T_1$ is the local existence time associated with the initial value $u(t_0)$ given by the local existence theorem for solutions. Define \(\tilde{v}(t) = u(t_0 + t)\) for \(t \in [0, T_1/2]\). Since \(t_0 + t \leq t_0 + T_1/2 = T_{u_0}^*\), the solution $\tilde{v}(t)$ is well-defined on $[0,T_1/2]$. As  system \eqref{eq:3} is invariant under time translation: if $u(t,x,y)$ is a solution to the generalized CH-KP equation \eqref{GCHKP}, then $u(t+t_0,x,y)$ is also a solution. Thus  $\tilde{v}(t)$ solves the system \eqref{eq:3} with initial condition $\tilde{v}(0) = u(t_0)$.

By the local existence theorem for solutions, given an initial value $u(t_0) \in X^s$, system \eqref{eq:3} admits a unique solution $\tilde{u}(t)$ on $[0,T_1]$ such that $$\tilde{u} \in C([0,T_1]; X^s(\mathbb{R}^2)) \cap C^1([0,T_1];X^{s-2}(\mathbb{R}^2)).$$  On the other hand, $\tilde{v}(t)$ is also a solution to system \eqref{eq:3} on $[0,T_1/2]$ that satisfies the same initial condition  $\tilde{v}(0) = u(t_0)$. Due to the uniqueness of the solution,  we have $$\tilde{u}(t) = \tilde{v}(t) = u(t_0 + t), \forall t \in [0,T_1/2].$$
At $t = T_1/2$, we find that $$\tilde{u}(T_1/2) = u(t_0 + T_1/2) = u(T_{u_0}^* - T_1/2 + T_1/2) = u(T_{u_0}^*).$$ This implies that the value of the solution $u$ at $t = T_{u_0}^*$ is given by $\tilde{u}$. Since $\tilde{u}$ is defined on the larger interval $[0,T_1]$,  it is  an extension of $u$ beyond $T_{u_0}^*$: $t_0 + T_1 = T_{u_0}^* - T_1/2 + T_1 = T_{u_0}^* + T_1/2 > T_{u_0}^*$, This contradicts the maximality of \(T_{u_0}^*\). Consequently, the initial hypothesis must be false, i.e.,
\[\int_0^{T_{u_0}^*} \|\nabla u(\tau)\|_{L^\infty(\mathbb{R}^2)}^2 \mathrm{d}\tau=\infty.\]

\textbf{Case 2:}  $s = 2.$ Substituting inequality \eqref{eq:62} into inequality \eqref{eq:44}, we have got
\begin{equation}\label{eq:68}
\begin{aligned}
&\frac{\mathrm{d}}{\mathrm{d}t}\left(\|\nabla u\|_{L^2}^2 + \|\nabla u_x\|_{L^2}^2\right)\\
&\quad\leq C(\gamma, g',M)\left(1+\|u_0\|_{L^2}+\|u_{0x}\|_{L^2}+2\|\nabla u\|_{L^\infty}\right)\\
&\qquad\times \left(\|\nabla u\|_{L^2}^2 + \|\nabla u_x\|_{L^2}^2\right).
\end{aligned}
\end{equation}
Applying Gr\"{o}nwall's inequality to inequality \eqref{eq:68}, one yields
\begin{equation}\label{eq:69}
\begin{aligned}
&\|\nabla u\|_{L^2}^2+ \|\nabla u_x\|_{L^2}^2 \leq \left(\|\nabla u_0\|_{L^2}^2 + \|\nabla u_{0x}\|_{L^2}^2\right)\times \\
& \quad\exp\left\{ C(\gamma, g',M) \left[(1+\|u_0\|_{L^2}+\|u_{0x}\|_{L^2})t + 2\int_0^t \|\nabla u\|_{L^\infty} \mathrm{d}\tau\right] \right\}.
\end{aligned}
\end{equation}
If $\int_0^{T_{u_0}^*} \|\nabla u(\tau)\|_{L^\infty(\mathbb{R}^2)}^2 \mathrm{d}\tau$ is bounded, then $\int_0^{T_{u_0}^*} \|\nabla u(\tau)\|_{L^\infty(\mathbb{R}^2)} \mathrm{d}\tau$ is also bounded. Then, for any $ t \in [0,T_{u_0}^*)$, we have
\begin{equation}\label{eq:70}
\|\nabla u\|_{L^2}^2 + \|\nabla u_x\|_{L^2}^2 \leq M_5(T_{u_0}^*),
\end{equation}
where
\begin{equation}\label{eq:71}
\begin{aligned}
M_5(T_{u_0}^*) &= \left(\|\nabla u_0\|_{L^2}^2 + \|\nabla u_{0x}\|_{L^2}^2\right) \exp\bigg\{ C(\gamma, g',M) \\
&\quad  \times\bigg[(1+\|u_0\|_{L^2}+\|u_{0x}\|_{L^2})T_{u_0}^* + 2\int_0^{T_{u_0}^*} \|\nabla u\|_{L^\infty} \mathrm{d}\tau \bigg] \bigg\}.
\end{aligned}
\end{equation}
Thus we have
\begin{equation}\label{eq:72}
\|\nabla u\|_{L^2} \leq \sqrt{M_5(T_{u_0}^*)}.
\end{equation}

Substituting inequality \eqref{eq:62} into inequality \eqref{eq:45}, it implies
\begin{equation}\label{eq:73}
\begin{split}
&\frac{\dd}{\dd t}\left( \|\nabla u\|_{L^2}^2 + \|\nabla u_x\|_{L^2}^2 + \|\nabla u_{xx}\|_{L^2}^2 \right)\\
&\quad\leq C\left( \gamma, g', g'', M \right)\big[ \|\nabla u_x\|_{L^2} \\
&\qquad + \left( 1 + \|u_0\|_{L^2} + \|u_{0x}\|_{L^2} + \|\nabla u\|_{L^\infty} \right) \left( 1 + \|\nabla u\|_{L^\infty} \right)\big] \\
&\qquad \times \left( \|\nabla u\|_{L^2}^2 + \|\nabla u_x\|_{L^2}^2 + \|\nabla u_{xx}\|_{L^2}^2 \right) \\
&\quad\leq C\left( \gamma, g', g'', M\right) \big[\|\nabla u_x\|_{L^2} + 1 + \|u_0\|_{L^2} + \|u_{0x}\|_{L^2} \\
&\qquad + \left( 2 + \|u_0\|_{L^2} + \|u_{0x}\|_{L^2} \right) \|\nabla u\|_{L^\infty} + \|\nabla u\|_{L^\infty}^2 \big]\\
&\qquad \times \left( \|\nabla u\|_{L^2}^2 + \|\nabla u_x\|_{L^2}^2 + \|\nabla u_{xx}\|_{L^2}^2 \right).
\end{split}
\end{equation}
From inequality \eqref{eq:72},  inequality \eqref{eq:73} arrives
\begin{equation}\label{eq:74}
\begin{split}
&\frac{\dd}{\dd t}\left( \|\nabla u\|_{L^2}^2 + \|\nabla u_x\|_{L^2}^2 + \|\nabla u_{xx}\|_{L^2}^2 \right)\\
&\quad\leq C\left( \gamma, g', g'', M \right) \bigg[ \sqrt{M_5(T_{u_0}^*)}+ \|u_0\|_{L^2} + \|u_{0x}\|_{L^2}  \\
&\qquad  + 1 +  \left( 2 + \|u_0\|_{L^2} + \|u_{0x}\|_{L^2} \right) \|\nabla u\|_{L^\infty} + \|\nabla u\|_{L^\infty}^2 \bigg] \\
&\qquad \times \left( \|\nabla u\|_{L^2}^2 + \|\nabla u_x\|_{L^2}^2 + \|\nabla u_{xx}\|_{L^2}^2 \right).
\end{split}
\end{equation}
Applying G\"{o}nwall's inequality to inequality \eqref{eq:74}, we yield
\begin{equation}\label{eq:75}
\begin{split}
&\|\nabla u\|_{L^2}^2 + \|\nabla u_x\|_{L^2}^2 + \|\nabla u_{xx}\|_{L^2}^2\\
&\quad\leq \left( \|\nabla u_0\|_{L^2}^2 + \|\nabla u_{0x}\|_{L^2}^2 + \|\nabla u_{0xx}\|_{L^2}^2 \right) \times\\
&\qquad\exp\left\{C\left( \gamma, g', g'', M \right) \left[ \left( 2 + \|u_0\|_{L^2} + \|u_{0x}\|_{L^2} \right) \int_0^t \|\nabla u\|_{L^\infty} \dd\tau \right. \right. \\
&\qquad+ \left. \left. \int_0^t \|\nabla u\|_{L^\infty}^2 \dd\tau + \left( \sqrt{M_5(T_{u_0}^*)} + 1 + \|u_0\|_{L^2} + \|u_{0x}\|_{L^2} \right)t \right] \right\}.
\end{split}
\end{equation}
Then, for any \( t \in [0, T_{u_0}^*)\), we have
\begin{equation}\label{eq:76}
\|\nabla u\|_{L^2}^2 + \|\nabla u_x\|_{L^2}^2 + \|\nabla u_{xx}\|_{L^2}^2 \leq M_6(T_{u_0}^*),
\end{equation}
where
\begin{equation}\label{eq:77}
\begin{split}
M_6(T_{u_0}^*) &= \left( \|\nabla u_0\|_{L^2}^2 + \|\nabla u_{0x}\|_{L^2}^2 + \|\nabla u_{0xx}\|_{L^2}^2 \right) \\
&\quad \times \exp\left\{ C\left( \gamma, g', g'', M \right)  \left[ \left( 2 + \|u_0\|_{L^2} + \|u_{0x}\|_{L^2} \right) \int_0^{T_{u_0}^*} \|\nabla u\|_{L^\infty} \dd\tau \right. \right. \\
&\quad + \left. \left. \int_0^{T_{u_0}^*} \|\nabla u\|_{L^\infty}^2 \dd\tau + \left( \sqrt{M_5(T_{u_0}^*)} + 1 + \|u_0\|_{L^2} + \|u_{0x}\|_{L^2} \right) T_{u_0}^* \right] \right\}.
\end{split}
\end{equation}
Thus we have
\begin{equation}\label{eq:78}
\begin{split}
\|\nabla u\|_{L^2} \leq \sqrt{M_6(T_{u_0}^*)},\
\|\nabla u_x\|_{L^2} \leq \sqrt{M_6(T_{u_0}^*)},\
\|\nabla u_{xx}\|_{L^2} \leq \sqrt{M_6(T_{u_0}^*)}.
\end{split}
\end{equation}

Substituting inequalities \eqref{eq:62} and \eqref{eq:78} into inequality \eqref{eq:54},  we obtain
\begin{equation*}\label{eq:79}
\begin{split}
\frac{\dd}{\dd t} \|u\|_{X^s}^2
&\leq C\left( \gamma, g', g'', M \right) \bigg[ 2\sqrt{M_6(T_{u_0}^*)} + \left( \|\nabla u\|_{L^\infty} + \sqrt{M_6(T_{u_0}^*)} + 1 \right)  \\
&\quad \times  \left( \|\nabla u\|_{L^\infty} + 1 + \|u_0\|_{L^2} + \|u_{0x}\|_{L^2} \right) \bigg] \|u\|_{X^s}^2.
\end{split}
\end{equation*}
Thus, we have
\begin{equation*}\label{eq:80}
\begin{split}
\|u\|_{X^s}^2
&\leq \|u_0\|_{X^s}^2 \exp\bigg\{ C\left( \gamma, g', g'', M \right)  \bigg[ \int_0^t \|\nabla u\|_{L^\infty}^2 \dd\tau  \\
&\quad + \left( \sqrt{M_6(T_{u_0}^*)} + \|u_0\|_{L^2} + \|u_{0x}\|_{L^2} + 2 \right) \int_0^t \|\nabla u\|_{L^\infty} \dd\tau \\
&\quad +   \left( \left( \sqrt{M_6(T_{u_0}^*)} + 1 \right) \left( \|u_0\|_{L^2} + \|u_{0x}\|_{L^2} + 1 \right) + 2\sqrt{M_6(T_{u_0}^*)} \right) t \bigg] \bigg\}.
\end{split}
\end{equation*}
This shows that for any \( t \in [0, T_{u_0}^*) \),
\begin{equation*}\label{eq:81}
\|u\|_{X^s}^2 \leq M_7(T_{u_0}^*),
\end{equation*}
where
\begin{equation}\label{eq:82}
\begin{split}
M_7(T_{u_0}^*)
&= \|u_0\|_{X^s}^2 \exp\bigg\{ C\left( \gamma, g', g'', M \right)  \bigg[ \int_0^{T_{u_0}^*} \|\nabla u\|_{L^\infty}^2 d\tau   \\
&\quad + \left( \sqrt{M_6(T_{u_0}^*)} + \|u_0\|_{L^2} + \|u_{0x}\|_{L^2} + 2 \right) \int_0^{T_{u_0}^*} \|\nabla u\|_{L^\infty} d\tau \\
&\quad +   \left( \left( \sqrt{M_6(T_{u_0}^*)} + 1 \right) \left( \|u_0\|_{L^2} + \|u_{0x}\|_{L^2} + 1 \right) + 2\sqrt{M_6(T_{u_0}^*)} \right) T_{u_0}^* \bigg] \bigg\}.
\end{split}
\end{equation}
Then the norm of $u$ in the $X^s$ space is uniformly bounded before the global maximal existence time $T_{u_0}^*$.

Applying the same  argument for the solution extension as in \textbf{Case 1}, it can reach a contradiction. Thus, the assumption is invalid, i.e.,
\begin{equation*}\label{eq:83}
\int_0^{T_{u_0}^*} \|\nabla u(\tau)\|_{L^\infty(\mathbb{R}^2)}^2 \, \dd\tau = \infty.
\end{equation*}

Consequently, the conclusion of Theorem \ref{T1.1} holds for any $C^\infty$ function $g(u)$ with $g(0) = 0$. This completes the proof of Theorem \ref{T1.1}.

\section{Blow-up sufficient conditions of solution}\label{S3}
\subsection{The case of $g'\in L^\infty$}
  Since $g'\in L^\infty(\mathbb{R})$, there exists a constant $L>0$ such that $\|g'\|_{L^\infty}\le L.$

\textbf{The proof of Theorem \ref{T1.2}}

The proof  Theorem \ref{T1.2} proceeds in four steps: first, defining characteristic lines and associated functions; second, deriving  the corresponding equations along these lines; third, estimating the $L^\infty$ norm of the residual term; and finally, establishing the blow-up conclusion by the Riccati inequality.

\begin{proof}
\textbf{Step 1: Defining the trajectories and the functions along trajectories}

Choose initial values $(x_0, y_0)$ such that $\partial_x u_0(x_0, y_0) = m_0$. Fix $y_0$ and define the characteristic curve $x = q(t, x_0, y_0)$ by
\[\frac{\mathrm{d}q(t,x_0,y_0)}{\mathrm{d}t} = \gamma u(t,q(t,x_0,y_0),y_0),\ q(0,x_0,y_0)= x_0.\]
Along this characteristic line, we define
\[w(t) = \partial_x u(t,x(t), y_0).\]
Clearly, $w(0)=m_0$.

\textbf{Step 2: Deriving the evolution equation for the derivative along the characteristics}

We commence by rewriting the first equation of  system \eqref{eq:3} in the form:
\begin{equation}\label{eq:84}
\begin{split}
u_t + \gamma u u_x + G \ast \partial_x F + G \ast v_y = 0,\ F = \frac{1}{2} g(u) + \frac{\gamma}{2} u_x^2 - \frac{\gamma}{2} u^2.
\end{split}
\end{equation}
Differentiating \eqref{eq:84}  with respect to $x$ yields
\begin{equation}\label{eq:85}
\partial_x u_t + \gamma\left( u_x \partial_x u + u u_{xx} \right) + \partial_x \left( G \ast \partial_x F \right) + \partial_x \left( G \ast v_y \right) = 0.
\end{equation}
Substituting $w$ into  equation \eqref{eq:85}, this gives
\begin{equation}\label{eq:86}
\partial_t w + \gamma\left( w^2 + u w_x \right) + \partial_x \left( G \ast \partial_x F \right) + \partial_x \left( G \ast v_y \right) = 0.
\end{equation}

Employing the material derivative $\frac{\mathrm{D}}{\mathrm{D}t} = \partial_t  + \gamma u \cdot \partial_x $ along the characteristic lines, we derive the evolution equation
\begin{equation}\label{eq:87}
\frac{\mathrm{D}w}{\mathrm{D}t} = -\gamma w^2 - R(t),\quad w(0) = m_0,
\end{equation}
where the residual term is defined as
$$R(t)=\partial_x (G* \partial_x F) + \partial_x(G*v_y).$$

\textbf{Step 3: Estimating the residual term \( R(t) \)}

We decompose the residual term as $R(t)=R_1(t)+R_2(t)$,
where $R_1(t) = \partial_x \left( G \ast \partial_x F \right),\ R_2(t) = \partial_x \left( G \ast v_y \right)$.

\textbf{ Estimate of \( R_1 \)}.

From the condition \( \|g'\|_{L^{\infty}} \leq L \) and Lemma \ref{L2.4}, it follows that
\begin{equation}\label{eq:88}
\|g(u)\|_{H^s} \leq C\left(L, \|u\|_{L^\infty}\right) \|u\|_{H^s} \leq C\left(L, \|u\|_{X^s}\right) \|u\|_{X^s}.
\end{equation}
Combining this with Lemma \ref{L2.2}, we obtain
\begin{equation}\label{eq:89}
\|F\|_{H^s} = \left\| \frac{1}{2}g(u) + \frac{\gamma}{2}u_x^2 - \frac{\gamma}{2}u^2 \right\|_{H^s}
\leq C\left( L, \|u\|_{X^s}\right)\|u\|_{X^s}+C(\gamma)\|u\|_{X^s}^2.
\end{equation}
By Lemma \ref{L2.1}(c), we have $\| \partial_x (G \ast f)\|_{H^s(\mathbb{R}^2)} \leq \|f\|_{H^s(\mathbb{R}^2)} $. Combining this with \eqref{eq:89}, we get
\begin{equation}\label{eq:90}
\begin{aligned}
\|R_1\|_{L^\infty} &= \|\partial_x (G \ast \partial_x F)\|_{L^\infty} \leq \|\partial_x (G \ast \partial_x F)\|_{H^{s-1}} \\
&\leq  \|F\|_{H^s} \leq C\left( L, \|u\|_{X^s}\right)\|u\|_{X^s}+C(\gamma)\|u\|_{X^s}^2.
\end{aligned}
\end{equation}

\textbf{The estimate of \( R_2 \)}. From the relation $u_y=v_x$, we deduce $v_{y}=\partial^{-1}_x u_{yy}$. Applying first inequality of Lemma \ref{C2.1} $$\| u\|_{L^{\infty}}^2\leq C\| u\|_{L^{2}}^{\frac{1}{2}}\| u_{x}\|_{L^{2}}^{\frac{1}{2}}\| u_{y}\|_{L^2}^{\frac{1}{2}}\| u_{xy}\|_{L^2}^{\frac{1}{2}},$$ we obtain, for $s\ge3$,

\begin{equation*}
\begin{aligned}
\|v_y\|_{L^\infty}^2 &=\| \partial_x^{-1}u_{yy}\|_{L^{\infty}}^2 \\
&\leq  C\|\partial_x^{-1}u_{yy} \|_{L^{2}}^{\frac{1}{2}}\| \partial_x^{-1}u_{xyy}\|_{L^{2}}^{\frac{1}{2}}\| \partial_x^{-1}u_{yyy}\|_{L^2}^{\frac{1}{2}}\| \partial_x^{-1}u_{yyyx}\|_{L^2}^{\frac{1}{2}}\\
&\leq C\|\partial_x^{-1}u_{yy} \|_{L^{2}}^{\frac{1}{2}}\| \partial_x^{-1}u_{yyy}\|_{L^2}^{\frac{1}{2}}\|u_{yy}\|_{L^{2}}^{\frac{1}{2}} \|u_{yyy}\|_{L^2}^{\frac{1}{2}}\\
&\leq C\|\partial_x^{-1}u\|_{H^3}\|u\|_{H^3}\le C\|u\|^2_{X^3}\le C\|u\|_{X^s}^2.
\end{aligned}
\end{equation*}
 Combining the inequality above with Lemma \ref{L2.1}(a), we conclude that
\begin{equation}\label{eq:91}
\begin{aligned}
\|R_2\|_{L^\infty}=\| \partial_x (G * v_y) \|_{L^\infty} \leq C \|  v_y \|_{L^\infty}\leq C \| u\|_{X^s}.
\end{aligned}
\end{equation}

From \eqref{eq:25}, there exists \( T > 0 \) such that for any \( t \in [0,T]\), we have \( \sup\limits_{\tau \in [0,t]} \|u(\tau)\|_{X^s} \le 2\|u_0\|_{X^s}\). Combining this with \eqref{eq:90} and \eqref{eq:91}, we see that there exists a constant \( K = C\left(\gamma, L, \|u_0\|_{X^s}\right) >0\) such that for $t\in[0,T]$, we have
\begin{equation*}
|R(t)| \le \|R_1\|_{L^\infty} + \|R_2\|_{L^\infty} \le K.
\end{equation*}

\textbf{Step 4: Deriving the blow-up conclusion from the Riccati inequality}

Along the characteristic line, it holds that
\begin{equation}\label{eq:93}
\frac{\mathrm{D}w}{\mathrm{D}t} = -\gamma w^2 - R(t), \ |R(t)| \le K, \ w(0) = m_0.
\end{equation}
Thus
\begin{equation}\label{eq:94}
\frac{\mathrm{D}w}{\mathrm{D}t} \le -\gamma w^2 + K.
\end{equation}
Let \( \psi(t) = -w(t) \), so that $\psi(0)=-m_0$, and
\begin{equation}\label{eq:95}
\frac{\mathrm{D}\psi}{\mathrm{D}t} = -\frac{\mathrm{D}w}{\mathrm{D}t} \ge \gamma w^2 - K = \gamma \psi^2 - K.
\end{equation}

If \( w_0 =m_0< -\sqrt{\frac{K}{\gamma}} \), then $\psi(t)$  is monotonically increasing and \( \psi(t) > \sqrt{\frac{K}{\gamma}} \). Thus
\begin{equation}\label{eq:96}
\frac{\dd\psi}{\gamma \psi^2 - K} \ge \dd t.
\end{equation}
Simultaneously integrating both sides and simplifying yields
\begin{equation}\label{eq:97}
\psi(t) \ge \frac{2\sqrt{K/\gamma}}{1 - \frac{\psi_0 - \sqrt{K/\gamma}}{\psi_0 + \sqrt{K/\gamma}} \cdot e^{2\sqrt{K\gamma}\, t}} - \sqrt{K/\gamma}.
\end{equation}
Thus, as \( t \to \left( \frac{1}{2\sqrt{K\gamma}} \ln\left( \frac{\psi_0 + \sqrt{K/\gamma}}{\psi_0 - \sqrt{K/\gamma}} \right) \right)^- \), we have
\begin{equation}\label{eq:98}
\left( \frac{2\sqrt{K/\gamma}}{1 - \frac{\psi_0 - \sqrt{K/\gamma}}{\psi_0 + \sqrt{K/\gamma}} \cdot e^{2\sqrt{K\gamma}\, t}} - \sqrt{K/\gamma} \right) \to +\infty.
\end{equation}
Thus \( \psi(t) \to +\infty \), and consequently \( \omega(t) = \partial_x u(t) \to -\infty \), meaning the solution blows up. Let
\begin{align}\label{eq:99}
T^* = \frac{1}{2\sqrt{K\gamma}} \ln\left( \frac{\psi_0 + \sqrt{K/\gamma}}{\psi_0 - \sqrt{K/\gamma}} \right) = \frac{1}{2\sqrt{K\gamma}} \ln\left( \frac{\sqrt{\gamma}m_0 - \sqrt{K}}{\sqrt{\gamma}m_0 + \sqrt{K}} \right).
\end{align}
Thus, as \( t \to (T^*)^- \), \( \|\partial_x u(t)\|_{L^\infty} \to \infty \). The theorem is proved.
\end{proof}

\textbf{The proof of Theorem \ref{T1.3}}

In \cite{WKZ2022}, a weighted blow-up theorem was established for the specific cases $g(u)=u^3$ and $g(u)=u^4$, yielding a sufficient condition for finite-time blow-up of solutions to system \eqref{eq:3}, namely that the initial data satisfy a certain inequality. Theorem \ref{T1.3} of this paper shows that for a general smooth function $g(u)$, the finite-time blow-up of solutions to system \eqref{eq:3} holds if the initial velocity and $g$ satisfy the prepared conditions.

\begin{proof}
As demonstrated in the proof of Theorem 4 in \cite{WKZ2022}, after differentiating both sides of equation \eqref{GCHKP} with respect to \( x \), then multiplying by \(\varphi(y)\), integrating the resulting expression with respect to \( y \) over \(\mathbb{R}\)along the trajectory,  we have already obtained
\begin{multline}\label{eq:100}
\frac{\mathrm{d}}{\mathrm{d}t} \int_{\mathbb{R}} u_{x}\big(t, q(t, x, y), y\big)\varphi(y)\mathrm{d}y = \int_{\mathbb{R}} \big(u_{tx} + \gamma u_{xx} u\big)\varphi\mathrm{d}y= - \frac{\gamma}{2} \int_{\mathbb{R}} u_x^2\varphi\mathrm{d}y \\
+ \int_{\mathbb{R}} \left[ \frac{1}{2}g(u) - \frac{\gamma}{2} u^2 - G * \left( \frac{1}{2}g(u) + \frac{\gamma}{2} u_x^2 - \frac{\gamma}{2} u^2 \right) \right] \varphi\mathrm{d}y - \int_{\mathbb{R}} (G * u) \varphi''\mathrm{d}y.
\end{multline}
Let \( M_1(t) = \int_{\mathbb{R}} u_{x}\big(t, q(t, x, y), y\big)\varphi(y)\mathrm{d}y \), then we have
\begin{equation}\label{eq:101}
\begin{split}
\frac{\dd M_1(t)}{\dd t} &= -\frac{\gamma}{2} \int_{\mathbb{R}} u_x^2 \varphi \, dy - \int_{\mathbb{R}} G*u \varphi'' \, \dd y \\
&\quad + \int_{\mathbb{R}} \left[ \frac{1}{2} g(u) - \frac{\gamma}{2} u^2 - G* \left( \frac{1}{2} g(u) + \frac{\gamma}{2} u_x^2 - \frac{\gamma}{2} u^2 \right) \right] \varphi \, \dd y.
\end{split}
\end{equation}
The literature \cite{WKZ2022} has obtained the following estimation formulas:
\begin{align}
&-\int_{\mathbb{R}} \frac{\gamma}{2} u^2 \varphi \, \dd y\leq \frac{\gamma}{2} \| \varphi \|_{L^\infty} \| u \|_{L_x^{\infty}L_y^2}^2 \leq \frac{\gamma}{2} \| \varphi \|_{L^\infty} \| u \|_{L^2} \| u_x \|_{L^2} \leq \frac{\gamma}{2} E(u_0) \| \varphi \|_{L^\infty},\label{eq:102}\\
&-\frac{\gamma}{2} \int_{\mathbb{R}} u_x^2 \varphi \, \dd y \leq -\frac{\gamma}{2} \left( \int_{\mathbb{R}} u_x \varphi \, dy \right)^2 = -\frac{\gamma}{2} M_1^2(t),\label{eq:103}\\
&\int_{\mathbb{R}} G* \left( \frac{\gamma}{2} u_x^2 - \frac{\gamma}{2} u^2 \right) \varphi \, \dd y \leq \| G* \left( \frac{\gamma}{2} u_x^2 - \frac{\gamma}{2} u^2 \right) \|_{L^1} \| \varphi \|_{L^\infty} \leq \gamma E(u_0) \| \varphi \|_{L^\infty},\label{eq:104}\\
&\int_{\mathbb{R}} G*u \varphi'' \, \dd y \leq \| G*u \|_{L_x^{\infty}L_y^2} \| \varphi'' \|_{L^2} \leq \| u \|_{L^2}^{\frac{1}{2}} \| u_x \|_{L^2}^{\frac{1}{2}} \| \varphi'' \|_{L^2} \leq E^{\frac{1}{2}}(u_0) \| \varphi'' \|_{L^2},\label{eq:105}
\end{align}
where $E(u) = \frac{1}{2} \int_{\mathbb{R}} (u^2 + u_x^2) \, \dd y.$

In what follows, we present an alternative estimation procedure for the key term, which  differs from that in \cite{WKZ2022}. Applying H\"{o}lder's inequality, we find
\begin{equation}\label{eq:106}
\begin{split}
\int_{\mathbb{R}} \frac{1}{2} g(u)\varphi \, \dd y
&\leq \frac{1}{2} \|g(u)\|_{L^2} \|\varphi\|_{L^2} \\
&\leq \frac{1}{2} \|g(u)\|_{H^s} \|\varphi\|_{L^2} \\
&\leq C(L,\|u_0\|_{X^s}) \|u_0\|_{X^s} \|\varphi\|_{L^2}.
\end{split}
\end{equation}
Then, by the H\"{o}lder's inequality and the Young inequality for convolution, we obtain
\begin{equation}\label{eq:107}
\begin{split}
\int_{\mathbb{R}} G*\left(\frac{1}{2} g(u)\right)\varphi \, \dd y
&\leq \left\| G*\left(\frac{1}{2} g(u)\right) \right\|_{L^2} \|\varphi\|_{L^2} \\
&\leq \frac{1}{2} \|G\|_{L^1} \|g(u)\|_{L^2} \|\varphi\|_{L^2} \\
&\leq C(L,\|u_0\|_{X^s}) \|u_0\|_{X^s} \|\varphi\|_{L^2}.
\end{split}
\end{equation}

Substituting the inequalities \eqref{eq:102} through \eqref{eq:107} into equation \eqref{eq:101}, one observes
\begin{equation}\label{eq:108}
\begin{split}
\frac{\dd M_1(t)}{\dd t} &= \int_{\mathbb{R}} \left[ \frac{1}{2}g(u) - \frac{\gamma}{2}u^2 - G*\left( \frac{1}{2}g(u) + \frac{\gamma}{2}u_x^2 - \frac{\gamma}{2}u^2 \right) \right] \varphi \dd y \\
&\quad - \frac{\gamma}{2} \int_{\mathbb{R}} u_x^2 \varphi \dd y - \int_{\mathbb{R}} G*u \varphi'' \dd y \\
&\leq -\frac{\gamma}{2} M_1^2(t) + C(L,\|u_0\|_{X^s}) \|u_0\|_{X^s} \|\varphi\|_{L^2} + \frac{\gamma}{2}E(u_0) \|\varphi\|_{L^\infty} \\
&\quad + \gamma E(u_0) \|\varphi\|_{L^\infty} + E^{1/2}(u_0) \|\varphi''\|_{L^2} \\
&\leq -\frac{\gamma}{2} M_1^2(t) + C_3^2,
\end{split}
\end{equation}
where
\begin{equation}\label{eq:109}
C_3^2 = C(L,\|u_0\|_{X^s}) \|u_0\|_{X^s} + \frac{3}{2}\gamma E(u_0) \|\varphi\|_{L^\infty} + E^{1/2}(u_0) \|\varphi''\|_{L^2}.
\end{equation}

By the assumption \( M_1(0) < -\sqrt{\frac{2}{\gamma}} C_3 \). By continuity, it follows that for all \( t < T_0 \), we have
\( M_1(t) < -\sqrt{\frac{2}{\gamma}} C_3 \). Solving inequality \eqref{eq:108}, we know that
\begin{equation}\label{eq:110}
\frac{M_1(0) + \sqrt{\frac{2}{\gamma}} C_3}{M_1(0) - \sqrt{\frac{2}{\gamma}} C_3} e^{\sqrt{2\gamma}C_3\, t} - 1 \leq \frac{2 \sqrt{\frac{2}{\gamma}} C_3}{M_1(t) - \sqrt{\frac{2}{\gamma}} C_3} \leq 0.
\end{equation}
Therefore, there exists a time \( T_0 \) such that
\begin{equation}\label{eq:111}
\begin{split}
T_0 &< \frac{1}{\sqrt{2\gamma} \, C_3} \ln \frac{M_1(0) - \sqrt{\frac{2}{\gamma}} C_3}{M_1(0) + \sqrt{\frac{2}{\gamma}} C_3}\\
&=\frac{1}{\sqrt{2\gamma} C_3} \ln \frac{\sqrt{\gamma} \int_{\mathbb{R}} u_{0x}(x_0,y)\varphi(y)\mathrm{d}y - \sqrt{2}C_3}{\sqrt{\gamma} \int_{\mathbb{R}} u_{0x}(x_0,y)\varphi(y)\mathrm{d}y + \sqrt{2}C_3} < \infty.
\end{split}
\end{equation}
and it follows that \( \lim_{t \to T_0} M_1(t) = -\infty \). Thus Theorem \ref{T1.3} is proven. \end{proof}

\subsection{The case of $g$ satisfies \eqref{cofg}}
 If the function \(g\) satisfies the polynomial growth condition \eqref{cofg}, these arguments for Theorem \ref{T1.2} and Theorem \ref{T1.3} are still valid.

In fact, as demonstrated in the proof of Theorem \ref{T1.1}, when \(s \geq 2\), for any \(t \in [0, T]\), we have \[\|u\|_{X^s(\mathbb{R}^2)} \leq 2\|u_0\|_{X^s(\mathbb{R}^2)}.\] When $s \geq 2$, the embedding $X^s(\mathbb{R}^2)  \hookrightarrow L^{\infty}(\mathbb{R}^2)$ implies $$\|u\|_{L^\infty(\mathbb{R}^2)} \leq C\|u\|_{X^s(\mathbb{R}^2)} \leq C\|u_0\|_{X^s(\mathbb{R}^2)}$$ then
\begin{equation}\label{eq:122}
|g'(u)| \leq C_1 |u|^{\alpha}+C_2 \leq C_1 \|u\|_{L^{\infty}(\mathbb{R}^2)}^{\alpha}+C_2 \leq C_1' \|u_0\|_{X^s(\mathbb{R}^2)}^{\alpha}+C_2\le C,
\end{equation}
where the constant $C$ is independent of $u$. This shows that $g^\prime\in L^\infty,$ therefore, both Theorem \ref{T1.2} and Theorem \ref{T1.3}  hold.

\section{Liouville-type Theorem}\label{S4}

 The literature \cite{GLL2021} proves the Liouville-type property for a specific quadratic nonlinear term \ $g(u) = 3u^2$. In this section, we generalize this result to arbitrary smooth nonlinear terms satisfying the conditions  $g(u) \geq \gamma u^2$, and $g(u) > \gamma u^2$, for $u\neq0$ as follows.
\begin{thm}\label{T1.4} Let $u = u(t,x,y)$ be a non-trivial solution to the system \eqref{eq:3} in the space $C([0,T];X^s(\mathbb{R}^2)) \cap C^1([0,T];X^{s-2}(\mathbb{R}^2)), s \geq 2$, and assume $g$ satisfies
\begin{align}\label{eq:10}
g(u) \geq \gamma u^2,
\end{align}
and that $g(u)>\gamma u^2$ for all $u \neq 0$. Then there exists no open set $W \subseteq [0,T] \times \mathbb{R}$ such that  $$u(t,x,y) = 0, \forall (t,x) \in W,\ y \in \mathbb{R}.$$
\end{thm}
\begin{proof}
Using the framework established in \cite{GLL2021}, we prove Theorem \ref{T1.4} by contradiction. Suppose there exists an open set \( W \subseteq [0,T] \times \mathbb{R} \) such that for any \((t,x) \in W\) and \( y \in \mathbb{R} \), we have \( u(t,x,y) = 0 \).

First, integrating the system  \eqref{eq:3} with respect to \( y \) over \( \mathbb{R} \) yields
\begin{equation}\label{eq:112}
\int_{\mathbb{R}} \left( u_t + \gamma u u_x \right) \dd y = -\int_{\mathbb{R}} G*\partial_x \left( \frac{1}{2}g(u) + \frac{\gamma}{2}u_x^2 - \frac{\gamma}{2}u^2 \right) \dd y.
\end{equation}
It follows from the assumption that
\begin{equation}\label{eq:113}
u_t + \gamma u u_x = 0,\quad \forall\ (t,x) \in W,\ y \in \mathbb{R},
\end{equation}
\begin{equation}\label{eq:114}
\frac{1}{2}g(u) + \frac{\gamma}{2}u_x^2 - \frac{\gamma}{2}u^2 = 0,\quad \forall\ (t,x) \in W,\ y \in \mathbb{R}.
\end{equation}
Since \(W\) is an open set, there exists \(t^* \in [0,T],\ I \subseteq [c,d],\ c < d\) such that \(\{t^*\} \times I \subseteq W\). Define
\begin{equation}\label{eq:115}
\begin{split}
p(x,y) &= G_x * \left( \frac{1}{2}g(u) + \frac{\gamma}{2}u_x^2 - \frac{\gamma}{2}u^2 \right)(t^*,x,y) \\
&= -\frac{1}{2} \operatorname{sgn}(\cdot) e^{-|\cdot|} \ast \left( \frac{1}{2}g(u) + \frac{\gamma}{2}u_x^2 - \frac{\gamma}{2}u^2 \right)(t^*, \cdot, y),
\end{split}
\end{equation}
\begin{equation}\label{eq:116}
q(x,y) = \left( \frac{1}{2}g(u) + \frac{\gamma}{2}u_x^2 - \frac{\gamma}{2}u^2 \right)(t^*,x,y).
\end{equation}
The regularity of functions \( u, g \) guarantees the continuity of functions \( p, q \), then
\begin{equation}\label{eq:117}
\int_{\mathbb{R}} p(x,y) \dd y = 0,\quad q(x,y) = 0,\quad \forall\ x \in [c,d].
\end{equation}

From the condition \( g(u) \ge \gamma u^2 \), we obtain \( q(x,y) \ge 0 \). By the condition, for all \( u \neq 0 \), we have \( g(u) > \gamma u^2 \), implying \( q(x,y) = 0 \) if and only if \( u = 0 \). Thus, when \(g(u) \ge \gamma u^2\), the equation \(\frac{1}{2}g(u) + \frac{\gamma}{2}u_x^2 - \frac{\gamma}{2}u^2 = 0\) is equivalent to \(u = 0\) (i.e., it admits only the trivial solution).

For any \( z \notin [c,d] \), we have
\begin{equation}\label{eq:118}
-\operatorname{sgn}(d - z)e^{-|d - z|} > -\operatorname{sgn}(c - z)e^{-|c - z|},
\end{equation}
Then  given \( y \in \mathbb{R}\), we have
\begin{equation}\label{eq:119}
\begin{split}
p(d,y) &= -\frac{1}{2} \int_{-\infty}^{+\infty} \operatorname{sgn}(d - z)e^{-|d - z|} q(z,y) \dd z \\
&\ge -\frac{1}{2} \int_{-\infty}^{+\infty} \operatorname{sgn}(c - z)e^{-|c - z|} q(z,y) \dd z = p(c,y),
\end{split}
\end{equation}
with equality holds if and only if \(\forall\ z \in \mathbb{R},\ q(z,y)=0 \). By the first equality of \eqref{eq:117}, it follows \[ \int_{\mathbb{R}} \left( p(d,y)-p(c,y) \right) dy = 0. \] Furthermore, since \( p(d,y)-p(c,y) \ge 0 \), we conclude that \[ p(d,y)-p(c,y)=0, \forall y\in \mathbb{R},\] hence \( \forall\ (y, z) \in \mathbb{R}^2,\ q(z,y)=0\), i.e.
\begin{equation}\label{eq:120}
\left( \frac{1}{2}g(u) + \frac{\gamma}{2}u_x^2 - \frac{\gamma}{2}u^2 \right)(t^*,z,y) = 0,\quad \forall\ (y, z) \in \mathbb{R}^2.
\end{equation}
This implies that \(\ u(t^*,z,y)=0,
\forall\ (y, z) \in \mathbb{R}^2.\)

By the local well-posedness and uniqueness (see \cite{WKZ2022}), solutions that agree at some time must coincide for all times. Hence $ u(t,x,y) \equiv 0$ for all $t\in[0,T]$ and $x,y\in \mathbb{R}$, contradicting the assumption that $u$ is non-trivial. This completes the proof.
 \end{proof}
In short, under the conditions $g(u) \geq \gamma u^2$ and $g(u) > \gamma u^2$, for $u\neq0$, no non-trivial solution can vanish identically on any nonempty open space-time cylinder. In other words, the solutions enjoy Liouville-type property.

\section*{Acknowledgments}  X. Ke is supported by the Youth Innovation Exploration Fund (NSFRF2502086) under the Special Project for Basic Scientific Research Business Expenses (Natural Sciences) at Henan Polytechnic University, J. Wang is supported by the Scientific and Technological Research Project (GJJ2409301) of the Jiangxi Provincial Department of Education, A. Zang is supported by the National Natural Science Foundation of China (Nos. 12261093), by  Jiangxi Provincial Natural Science Foundation  (No. 20224ACB201004).
\section*{Data availability}
\quad No data was used for the research described in the article.

\end{document}